%% file: final_draft.tex
\documentclass[12pt,reqno]{article}
\input{header.tex}

\title{Homological Nielsen Realization for\\the Manifolds $\#_n\cp^2$}
\date{\vspace{-5ex}}
\author{Ethan Pesikoff\thanks{Department of Mathematics, University of Chicago. epesikoff@uchicago.edu.   \href{https://math.uchicago.edu/~epesikoff/}{Personal webpage}}}

\bibliography{references} 

\begin{document}

\parskip0pt
\parindent10pt
\linespread{1.1}

\makeatletter
\renewcommand{\@fnsymbol}[1]{\@arabic{#1}}
\makeatother

\maketitle
\begin{abstract}
    \noindent Given a smooth, oriented, simply-connected $4$-manifold $M$, the homological Nielsen realization problem asks: when does a finite group of isometries $G\leq O(H_2(M;\Z))$ preserving the intersection form lift isomorphically to a finite group of orientation-preserving diffeomorphisms?  We study this question for the smooth, positive-definite 4-manifolds $M_n\coloneqq\#_n\cp^2$.  Even though every isometry of $H_2(M_n;\Z)$ is induced by some orientation-preserving diffeomorphism, not necessarily of finite order, we show that Nielsen realization is sparse: as $n\to\infty$, a random subgroup of $O(H_2(M_n;\mathbb Z))$ is asymptotically almost never realizable in $\diff^+(M_n)$; the same is true for random odd order elements of $O(H_2(M_n;\Z))$.  We present both positive realization results in certain cases and a range of obstructions to realization in other cases. The proofs combine equivariant connected-sum constructions, fixed-point theory for group actions on 4-manifolds, finite group actions on surfaces, analytic combinatorics, and previous work of Hambleton--Tanase \cite{ht}.
\end{abstract}
\section{Introduction}

Let $M$ be a closed, oriented smooth 4-manifold.  Its orientation-preserving diffeomorphism group $\diff^+(M)$ acts on $H_2(M;\Z)$ respecting the intersection form \[Q:H_2(M;\Z)\times H_2(M;\Z)\to\Z.\]  This action gives a natural representation \begin{align*}
    \pi: \diff^+(M)&\to O(H_{M})\coloneqq O(H_2(M,\Z),Q)\\
    f\ \ \ \ &\mapsto   [f_*:H_2(M;\Z)\to H_2(M;\Z)].
\end{align*}
One version of the \textit{Nielsen realization problem} asks: for which finite groups $G\leq O(H_M)$ does $\pi|_G$ admit a group-theoretic section?  If such a section exists we say $G$ is \textit{realizable} in $\diff^+(M)$. We say $\phi\in O(H_M)$ is realizable if $\ang{\phi}\leq O(H_M)$ is realizable.

In this paper we consider this homological Nielsen realization problem for the closed, simply-connected, oriented, positive-definite 4-manifolds $M_n\coloneqq \#_n\cp^2$.  Baraglia showed \cite{bar} that for these manifolds $M_n$, the map $\pi$ is surjective for all $n\geq 1$; in fact, he shows that the map $\pi_0(\diff^+(M_n))\to O(H_{M_n})$ splits for all $n\geq 1$.  (This is substantially weaker than Nielsen realization, since finite order isometries need not be induced by finite order diffeomorphisms.  But in particular Nielsen realizability is a conjugacy invariant in $O(H_{M_n})$.)  Note that there is a standard basis for $H_2(M_n;\Z)$ given by the hyperplane classes $\{e_1,\dots,e_n\}$ in each copy of $\cp^2$.  In this basis, the intersection form $Q$ is given by the identity matrix $Id_n$; in other words, $O(H_{M_n})$ is given by the group of integral isometries $O(n,\Z)$ of $\R^n$ with its usual inner product.  It is standard that $O(H_{M_n})$ can be thought of as the group of signed $n$-by-$n$ permutation matrices on the basis $\{e_1,\dots,e_n\}$.  In this representation, let $S_n< O(H_{M_n})$ denote the subgroup of matrices with totally positive sign, and let $G_n\leq O(H_{M_n})$ denote the diagonal subgroup. In particular, $G_n\cong (\Z/2\Z)^n$.  We then have \[O(H_{M_n})\cong O(n,\Z)\cong G_n\rtimes S_n.\]

Recall that Nielsen realization for surfaces always holds: any finite group of mapping classes lifts to an isomorphic finite group of orientation-preserving diffeomorphisms \cite{kerckhoff}.  (See \Cref{rmk:modtop} below for why this is an appropriate analogue of our homological question in the surface case.)  Nielsen realization is also often true on del Pezzo and K3 surfaces \cite{blk3,s1,s3}.  Similarly in the case of our manifolds $M_n$, many subgroups $H\leq O(H_{M_n})$ are realizable in $\diff^+(M_n)$.  When $n=2$ or $n=3$ all cyclic subgroups of $O(H_{M_n})$ are realizable in $\diff^+(M_n)$, and when $n\leq 8$, all cyclic subgroups of $S_n\leq O(H_{M_n})$ are realizable in $\diff^+(M_n)$.  (See \Cref{prop:m23} and \Cref{cor:Sn-summary}.)  In fact, we justify the following observation below by \Cref{thm:large_rank} and in \Cref{sec:cyclic} by previous work of Hambleton--Tanase \cite{ht}.

\begin{obs*}
    At least $2^n$ elements of $O(H_{M_n})$ are realizable in $\diff^+(M_n)$, and at least $2^{2n-1}$ subgroups of $O(H_{M_n})$ are realizable in $\diff^+(M_n)$.  Even up to conjugacy, there are at least $c_1n^2\log(n)^2$ realizable elements of $O(H_{M_n})$ and at least $c_2n^3$ realizable subgroups, for some constants $c_1,c_2>0$.
\end{obs*}
\noindent   In contrast to all the above, the main result of this paper (\Cref{thm:mega-thm} below) is that the generic element (resp. subgroup) of $O(H_{M_n})$ is \textit{almost never realizable} in $\diff^+(M_n)$.  To make this precise, we use the following notions of randomness.\\

\noindent \textbf{Convention on Randomness.}  A \textit{random element} of a finite group $H$, satisfying a condition $Q$, means a uniformly chosen element of $\{h\in H : h \text{ satisfies } Q\}$, and similarly for random subgroups. When we say \textit{up to conjugacy}, we use the uniform measure on conjugacy classes (of elements or subgroups) in $H$ satisfying the given condition. For a sequence of such random choices depending on $n$, we say that an event occurs \textit{asymptotically almost never} if its probability tends to $0$ as $n\to\infty$. We say it occurs \textit{asymptotically with probability at most }$P$ if the corresponding limsup is at most $P$.\\

\noindent For a subgroup $H\leq G_n\leq O(H_{M_n})$, let the \textit{rank} of $H$, or $\rk(H)$, denote its $\F_2$-rank (equivalently, $\log_2(\abs{H})$).  In particular, $G_n\leq O(H_{M_n})$ has subgroups of rank $r$ for all $0\leq r\leq n$.  We now state our main theorem.\vspace{2cm}
\begin{theorem}\label{thm:mega-thm}{\bf(Main Theorem)} Let $n\geq 1$.
    \begin{enumerate}
        \item Random elements of odd order in $O(H_{M_n})$ are asymptotically almost never realizable in $\diff^+(M_n)$ as $n\to\infty$.  The same holds for random elements of odd order in $S_n\leq O(H_{M_n})$.
        \item Random subgroups of $O(H_{M_n})$ are asymptotically almost never realizable in $\diff^+(M_n)$.  The same is true up to conjugacy.
        \item  Let $H\leq G_n\leq O(H_{M_n})$.  If $\rk(H)\leq 2$ then $H$ is realizable in $\diff^+(M_n)$.  In contrast, $H$ is not realizable in $\diff^+(M_n)$ if either:\begin{enumerate}
            \item $n$ is odd and $\rk(H)\geq 4$, or
            \item $\rk(H)>8+\log_2 n$.
        \end{enumerate}
        \item Fix some $k\geq 4$.  Then a  random rank-$k$ subgroup of $G_n\leq O(H_{M_n})$ is realizable in $\diff^+(M_n)$ with asymptotic probability at most $2^{-k}$.
    \end{enumerate}
\end{theorem}

\noindent \Cref{thm:mega-thm} summarizes our main results but omits several refinements that we prove throughout the paper.  For example, we prove that random 2-subgroups of $S_n\leq O(H_{M_n})$ and random abelian subgroups of $O(H_{M_{2k+1}})$ are both asymptotically almost never realizable (\Cref{thm:2groups} and \Cref{thm:odd-ab}(3), respectively).

Our analysis splits naturally into three settings: first $G_n\leq O(H_{M_n})$, then $S_n\leq O(H_{M_n})$, and finally the whole group $O(H_{M_n})$.  In each setting, we first present constructions for realizing certain elements, then obstructions to realizing other groups satisfying some hypotheses, and finally in the case of asymptotic non-realizability statements, we show that these hypotheses are suitably generic.

\begin{rmk}
    By part 3 of \Cref{thm:mega-thm}, the diagonal subgroup $G_n\leq O(H_{M_n})$ exhibits a sharp transition in small rank: all subgroups of rank at most 2 are realizable, while for odd $n$, no subgroup of rank at least 4 is realizable. Thus rank 3 is the only unresolved boundary case in this regime. In \Cref{sec:gn_constr}, we explicitly construct an infinite family of realizable rank-3 subgroups.
\end{rmk}

\begin{rmk}\label{rmk:modtop}
    For $M$ any closed, oriented, simply-connected smooth 4-manifold, the map $\pi:\diff^+(M)\to O(H_M)$ factors through the topological mapping class group $\mpg(M)\coloneqq \pi_0(\homeo^+(M))$.  Work of Freedman \cite{freed}, Quinn, Gabai-Gay-Hartman-Krushkal-Powell, Kreck, and Perron (see \cite{gghkp} and the references therein) shows that the action of $\mpg(M)$ on $H_2(M;\Z)$ gives an isomorphism \[\phi:\mpg(M)\xrightarrow{\cong} O(H_{M}).\]  Consequently, the homological Nielsen realization of this paper is equivalent to the problem of lifting from $\mpg(M)$ to $\diff^+(M)$.
\end{rmk}

\begin{rmk}
    There are several other variations on the Nielsen realization problem.  For example, one might also wonder about lifting from $\mpg(M)$ to $\homeo^+(M)$, or from the smooth mapping class group $\pi_0(\diff^+(M))$ to $\diff^+(M)$.  For the manifolds $M_n$, working either with the smooth mapping class group or with $\homeo^+(M_n)$ leads into a deep world of exotic phenomena, which is not our primary focus here.  Both $\pi_0(\diff^+(M_n))$ and the \textit{Torelli group} \[T(M_n)\coloneqq \ker[\pi_0(\diff^+(M_n))\to O(H_{M_n})]\] remain poorly understood; examples of exotic phenomena on $M_n$ have been found \cite{ed2} but are far from being comprehensively classified.
\end{rmk}

\vspace{0.5cm}

\noindent{\bf Outline of the Paper.} In \Cref{sec:gn} we examine the lifting inside the reflection subgroup $G_n<O(H_{M_n})$.  We open with a series of increasingly general geometric constructions for realization and then present classical tools for obstructing realization.  We then prove refinements of parts 3 and 4 of \Cref{thm:mega-thm}.  In \Cref{sec:sn} we examine the lifting inside the permutation subgroup $S_n<O(H_{M_n})$.  We review previous work of Hambleton--Tanase \cite{ht}, which provides constraints and constructions for cyclic realization. We then prove the sparsity of realizable odd order elements of $S_n\leq O(H_{M_n})$ by combinatorial and analytic methods.  After proving an analogous sparsity result about 2-subgroups of $S_n\leq O(H_{M_n})$, we give realization constructions based on regular polytopes.  The section concludes with an obstruction to realizing large subgroups $S_k\leq S_n<O(H_{M_n})$. In \Cref{sec:other} we provide a geometric construction to realize certain cyclic subgroups which do not lie in either $G_n$ or $S_n<O(H_{M_n})$. In particular, cyclic Nielsen realization holds generally on $M_2$ and $M_3$.  Finally, we prove parts 1 and 2 of \Cref{thm:mega-thm}.

Obstructions to realization draw on several different sources: fixed-point and signature constraints in dimension four, including the Hirzebruch $G$-signature theorem and results of Edmonds \cite{ed} from Smith theory and representation theory; finite group actions on surfaces, including the $84(g-1)$ theorem; combinatorial and analytic estimates for partitions and permutations; group cohomology; and work of Hambleton--Tanase \cite{ht} proven using equivariant Yang-Mills moduli spaces.
\vspace{0.5cm}

\noindent{\bf Related work.} In dimension 4, recent work of S. Lee and Lee-Lewis-Raman studied Nielsen realization problems for Del Pezzo surfaces \cite{s1,s2,s3}.  In particular, our constructions below were first inspired by those in \cite{s1}.  Farb-Looijenga \cite{blk3} studied Nielsen realization problems for K3 surfaces.  Using techniques such as Seiberg-Witten invariants, Baraglia-Konno \cite{bk}, Konno \cite{konno}, Arabadji-Baykur \cite{ab}, and Konno-Miyazawa--Taniguchi \cite{kmt} have also achieved non-realization results in dimension 4.

Baraglia \cite{bar} studied the manifolds $n\cp^2\#_m\overline{\cp^2}$ and their mapping class groups; in particular, he gave an obstruction to Nielsen realizing a certain class of subgroups which inspired part of our proof of \Cref{thm:large_rank}.  Hambleton-Lee \cite{hl} and Hambleton--Tanase \cite{ht} studied certain rigidity phenomena on the manifolds $n\cp^2$, and their results provide essential input into the asymptotic non-realizability statements for $S_n<O(H_{M_n})$, \Cref{thm:asym,thm:2groups}.  In a slightly different direction, Edmonds \cite{ed2} studied locally linear, cyclic, pseudofree actions on $n\cp^2$, obstructing some and finding a non-smoothable example.

\vspace{0.5cm}

\noindent{\bf Acknowledgments.} I thank my advisor Benson Farb for his continuing support and encouragement throughout this project, for enlightening mathematical conversations, and for his many constructive comments on many previous drafts.  I also thank Faye Jackson and many other graduate students around me for helpful mathematical conversations and for answering my questions.  Thank you to Dan Margalit for comments on an earlier draft.  The author was partially supported by an NSF Graduate Research Fellowship under NSF grant no. 2140001.

\section{Realization in $G_n\leq O(H_{M_n})$}\label{sec:gn}

\noindent The main goal of this section is to prove the following theorem.  For an element $\phi\in G_n$, let $\ell(\phi)$ denote the number of $(-1)$s in its diagonal permutation matrix.  Recall that for $H\leq G_n\cong(\Z/2\Z)^n$, its \textit{rank} denotes its dimension as an $\F_2$-vector space
\begin{theorem}\label{thm:large_rank}
    Let $n\geq 1$, and let $H\leq G_n\leq O(H_{M_n})$. \begin{enumerate}
        \item If $\rk(H)\leq 2$, then $H$ is realizable in $\diff^+(M_n)$.  Equivalently, any subgroup of $G_n\leq O(H_{M_n})$ generated by at most two elements is realizable in $\diff^+(M_n)$.
        \item In contrast, $H$ is not realizable in $\diff^+(M_n)$ if any of the following hold.\begin{enumerate}
            \item $Rk(H) \geq 4$ and $n$ is odd.
            \item $Rk(H) \geq 4$ and there is some $\phi\in H$ with $(n-\ell(\phi))\equiv 1\mod 2$, where $\ell(\phi)$ denotes the number of $(-1)$s in its matrix representation.
            \item $Rk(H) > 8+\log_2(n)$.
        \end{enumerate} 
    \end{enumerate}
\end{theorem}

\noindent In particular, \Cref{thm:large_rank} implies part 3 of \Cref{thm:mega-thm}. We prove the first part of \Cref{thm:large_rank} by constructing explicit realizations with equivariant connect sums.  These methods generalize in \Cref{constr:3} below to construct certain realizable subgroups of $G_n\leq O(H_{M_n})$ of rank 3.  Also notice that by counting subspaces of $F_2^n$ of ranks 1 and 2, there are at least $2^n$ elements of $O(H_{M_n})$ and $2^{2n-1}$ subgroups of $O(H_{M_n})$ which are realizable in $\diff^+(M_n)$.  Since conjugation by $O(H_{M_n})$ acts on a subgroup $H\leq G_n\leq O(H_{M_n})$ by permuting the standard coordinate basis of $\F_2^n\cong G_n$, counting subspaces up to this $S_n$-action shows that at least $c_2 n^3$ subgroups of $O(H_{M_n})$ are realizable in $\diff^+(M_n)$.\\

\noindent Since the group $G_n\leq O(H_{M_n})$ satisfies one of the hypotheses of part 2 of \Cref{thm:large_rank} for $n\geq 4$, we have the following corollary.

\begin{cor}\label{cor:whole-group}
    For $n\geq 4$ the group $O(H_{M_n})$ is not realizable in $\diff^+(M_n)$.
\end{cor}

To conclude the section, we show the following result, which in particular implies part 4 of \Cref{thm:mega-thm}.

\begin{prop}\label{cor:gn}{\bf (Probabilistic non-realizability for subgroups of $G_n\leq O(H_{M_n})$)}\hfill
    \begin{enumerate}
        \item Fix some $k\geq 4$.  Then a random rank-$k$ subgroup of $G_n\leq O(H_{M_n})$ is realizable in $\diff^+(M_n)$ with asymptotic probability at most $2^{-k}$.
        \item A random subgroup of $G_n\leq O(H_{M_n})$ is asymptotically almost never realizable in $\diff^+(M_n)$.
        \item A random subgroup of $G_n\leq O(H_{M_n})$ with rank at most $\log_2 n$ is asymptotically almost never realizable in $\diff^+(M_n)$.  In fact, the probability that such a random subgroup is realizable in $\diff^+(M_n)$ is $O(n\inv)$.
    \end{enumerate}  
\end{prop}

In \Cref{sec:gn_constr}, we present three increasingly general constructions for realization and use them to prove the first part of \Cref{thm:large_rank}.  In \Cref{sec:tools}, we exposit a few useful tools for obstructing realization.  In \Cref{sec:pfthmlgrk}, we give proofs of \Cref{thm:large_rank}, \Cref{cor:whole-group}, and \Cref{cor:gn}.

\subsection{Three constructions for realization}\label{sec:gn_constr}
We present three realization constructions of increasing complexity and generality.  We then use them to prove the first part of \Cref{thm:large_rank}.
\begin{constr}\label{constr:1}
    Consider a single copy of $\cp^2$ with some fixed projective coordinate system $[X:Y:Z]$, and define the following self-diffeomorphisms.\begin{align*}
        f_X:&[X:Y:Z]\mapsto[-X:Y:Z]\\
        J:&[X:Y:Z]\mapsto[\bar{X}:\bar{Y}:\bar{Z}].
    \end{align*}
    We construct a model of $M_2$ with the following property: the map $f_X$ on one $\cp^2$-component extends naturally as a smooth involution which on the other copy acts as $J$.  Conversely, $J$ on the first copy should extend as $f_X$ on the second.
    
    We can see this construction in two ways; one is more abstract and provides a cleaner proof of its validity, and the other is a more concrete picture using local coordinates.  Both are useful to consider going forward, so we present both viewpoints.\\

\noindent \textbf{Method 1:} We follow an argument of S. Lee \cite{s1}, which we exposit here for clarity.
    \begin{lem}[\cite{s1}, Lemma 3.2]\label{lem:ser}
        Let $M$ and $N$ be closed, simply-connected oriented smooth 4-manifolds. Let $G_M\cong (\Z/2\Z)^2\leq \diff^+(M)$ and $G_N\cong (\Z/2\Z)^2\leq \diff^+(N)$, and fix a group isomorphism $\phi:G_M\to G_N$.  Suppose there exist $p\in Fix(G_M)\subseteq M$ and $q\in Fix(G_N)\subseteq N$, and for any $h\in G_M$, let $F_h$ and $F_{\phi(h)}$ denote the connected components of $p,q$ in $Fix(h),Fix(\phi(h))$, respectively.  Assume further that $F_h$ and $F_{\phi(h)}$ are each 2-dimensional.  For some pair of generators $f,g$ of $G_M$, assume that $F_f\cap F_g$ and $F_{\phi(f)}\cap F_{\phi(g)}$ are 1-dimensional.
    
            Then there exist diffeomorphisms $h\#\phi(h)$ of $\diff^+(M\#\bar{N})$ for all $h\in G_M$ such that $\ang{h\#\phi(h):h\in G_M}\cong G_M$, and \[[h\#\phi(h)]=([h],[\phi(h)])\in \mpg(M)\times\mpg(\bar N)\leq \mpg(M\# N).\]
        \end{lem}
    
        \begin{proof}[Proof sketch]
            Fixing $G_M$ (resp. $G_N$)-invariant metrics on $M$ (resp. $N$) gives tangential representations $\rho_M:G_M\into SO(T_pM)\cong SO(4)$ and $\rho_N:G_N\into SO(T_qN)\cong SO(4)$ by acting on the tangent space.  Note these representations are injective because an isometry of a compact manifold is determined by its action on a point and a frame.  By hypothesis, $\dim(T_pM^{G_M})=\dim(T_qN^{G_N})=1$, and so $-I$ is not in the image of either $\rho_M,\rho_N$.
    
            There is a unique faithful representation $(\Z/2\Z)^2\into SO(4)$ with image avoiding $-I$, up to conjugation in $SO(4)$.  Then $\rho_M,\rho_N$ are equivalent by an orientation-reversing isometry $T_pM\to T_qN$, or an orientation-preserving isometry $T_pM\to \overline{T_qN}$.  Hence this constructs the desired connect sum of maps with its sought properties.
        \end{proof}

        Now simply apply \Cref{lem:ser} to the manifolds $M=\cp^2$ and $N=\overline{\cp^2}$.  Let $G_M=\ang{f_X,J}$ and $G_N=\ang{J,f_X}$. Notice that although we defined the maps $f_X$ and $J$ on $\cp^2$, they are still perfectly well-defined and orientation preserving on the flipped manifold $\overline{\cp^2}$. Consider the isomorphism $G_M\to G_N$ induced by $f_X\mapsto J$, $J\mapsto f_X$.  Notice that
        \begin{align*}
            Fix(f_X)&=\{[0:Y:Z]\}\sqcup\{[1:0:0]\}\cong \cp^1\sqcup \{p\}\\
            Fix(J)&=\{[X:Y:Z]:X,Y,Z\in\R\}\cong\rp^2\\
            Fix(J)\cap Fix(f_X)&=\{[0:Y:Z]:Y,Z\in\R\}\cong S^1\\
        \end{align*}
        Letting $p$ and $q$ lie at $[0:1:0]$ in their respective coordinate systems then fulfills the requirements of \Cref{lem:ser} and produces the sought diffeomorphisms on $\cp^2\#\overline{\overline{\cp^2}}=\cp^2\#\cp^2=M_2$.

    \item \textbf{Method 2:} Construct a local model of a neighborhood of the connect sum, as follows. Around the point $[0:0:1]$, there is an affine chart in the $X_1,Y_1$ coordinates on one ball, and coordinates $X_2,Y_2$ locally in the other copy of $\cp^2$.  To form the sought connect sum, cut out this coordinate ball from each copy of $\cp^2$ and glue their boundary spheres by a particular orthogonal orientation-reversing map $\psi$. It is more convenient to construct $\psi$ by identifying the balls themselves, and then to restrict to the boundary.  On each ball, there is the real ordered basis $(\re X_i,\im X_i,\re Y_i,\im Y_i)$ in the respective coordinate systems.  Then $\psi$ is given by the matrix \[A=\begin{pmatrix}
        0&0&0&1\\
        0&1&0&0\\
        0&0&1&0\\
        1&0&0&0\\
    \end{pmatrix}.\]  That is, identify the two $X$-axes $\re X$ and $\im X$ on one side with the two imaginary axes $\im X$ and $\im Y$ on the other.  Since $\det(A)=-1$, it reverses orientation and gives the connect sum $M_2$.  From this local model, it is clear that $f_X$ on one side corresponds to complex conjugation on the other.

Note that along with $\psi$, there is another map $\tau$ given in the same local chart by the matrix \[T=\begin{pmatrix}
        -1&0&0&0\\
        0&1&0&0\\
        0&0&1&0\\
        0&0&0&1\\
    \end{pmatrix},\] identifies $f_X$ on one side again with $f_X$ on the other, and similarly for conjugation $J$.  Call the gluing by $A$ a nontrivial or `hinge' connect sum, the gluing by $T$ a `trivial' or `standard' connect sum. See \Cref{fig:hinge}.
\end{constr}

\begin{figure}[hbt!] 
    \centering
    \includegraphics[width=0.7\textwidth]{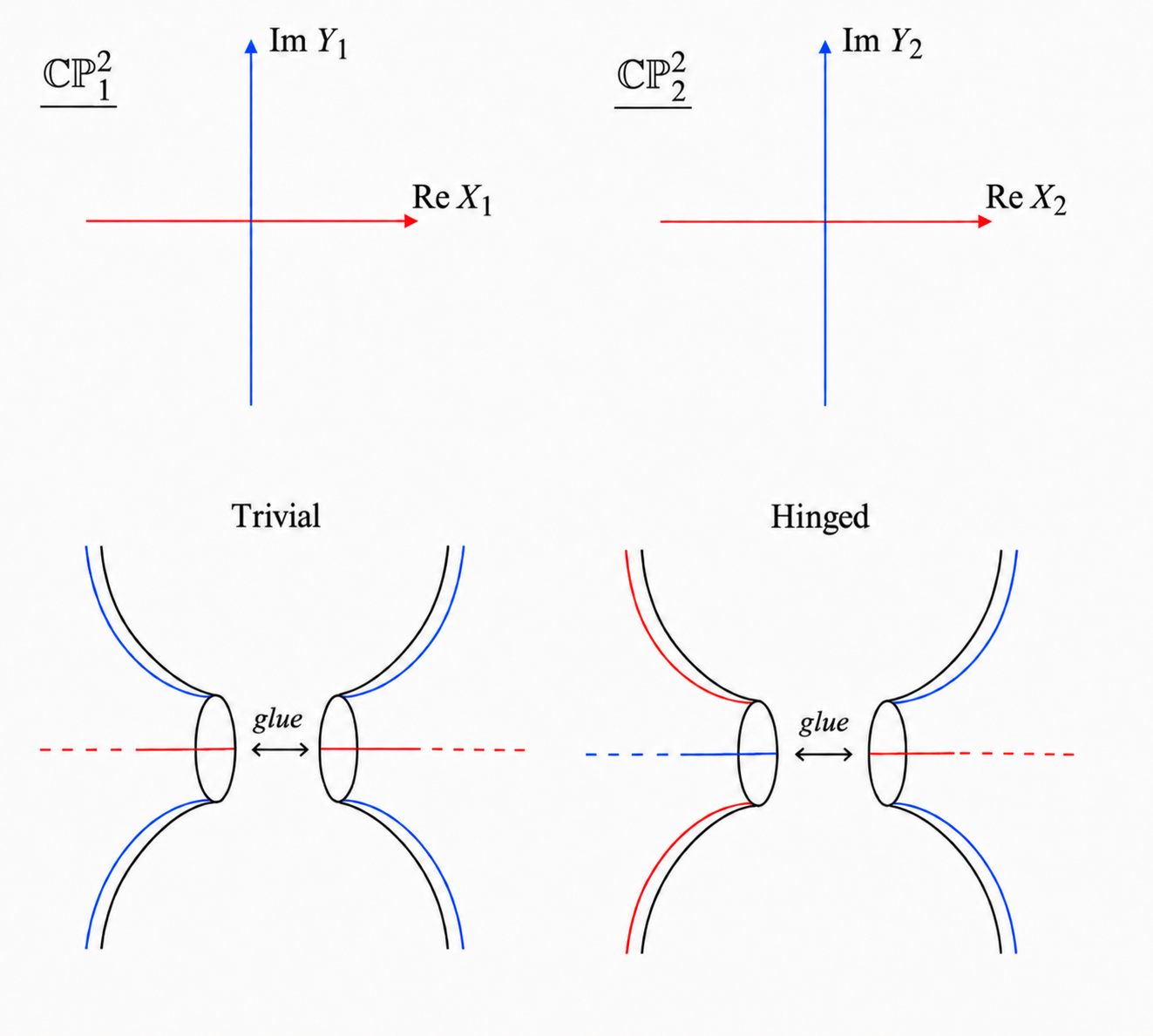}
    \caption{Trivial vs Hinge connect sum}
    \label{fig:hinge} 
\end{figure}

\vspace{5mm}

\begin{constr}\label{constr:2}
    Consider a single copy of $\cp^2$ with some fixed projective coordinate system $[X:Y:Z]$, and define the following self-diffeomorphisms.\begin{align*}
        f_X:&[X:Y:Z]\mapsto[-X:Y:Z]\\
        f_Y:&[X:Y:Z]\mapsto[X:-Y:Z]\\
        f_Z:&[X:Y:Z]\mapsto[X:Y:-Z]\\
        J:&[X:Y:Z]\mapsto[\bar{X}:\bar{Y}:\bar{Z}]\\
        J':&[X:Y:Z]\mapsto[\tilde{X}:\tilde{Y}:\tilde{Z}]
    \end{align*}where $\tilde X$ denotes reflection of $X$ about the imaginary axis (which is then conjugate to conjugation by multiplication by $i$). Define $\tilde Y$ and $\tilde Z$ similarly.

    Note that all five maps commute and that $f_X f_Y=f_Z$.  Also note that $f_X,f_Y,f_Z$ act by $+1$ on $H_2(\cp^2,\Z)$, and $J$ and $J'$ act by $-1$. The fixed set of the group $\ang{f_X,f_Y,J}$ is the three points $[1:0:0],[0:1:0],$ and $[0:0:1]$. We describe a construction of $M_4$ which attaches three additional copies of $\cp^2$ to one central copy at neighborhoods of those points $[1:0:0],[0:1:0],$ and $[0:0:1]$.
    
    The resulting construction has the following properties: label each of the three attached copies of $\cp^2$ as $M_X$, $M_Y$, and $M_Z$, according to their attaching points. Let $M_c$ denote the central component.  The maps $f_X,f_Y,f_Z,J$ extend to the manifold $M_4$ as follows. \begin{enumerate}
        \item $f_X$ extends as conjugation on $M_Y$, as $J'$ on $M_Z$, and as $f_X$ on $M_X$.
        \item $f_Y$ extends as conjugation on $M_Z$, as $J'$ on $M_X$, and as $f_Y$ on $M_Y$.
        \item $f_Z$ extends as conjugation on $M_X$, as $J'$ on $M_Y$, and as $f_Z$ on $M_Z$.
        \item $J$ extends to $f_X$ on $M_X$, $f_Y$ on $M_Y$, and $f_Z$ on $M_Z$.
    \end{enumerate}
    In particular, in the basis $([M_c],[M_X],[M_Y],[M_Z])$ where $[\cdot]$ denotes the hyperplane class of that copy of $\cp^2$, these functions act by \begin{align*}
        &(f_X)_*=\begin{pmatrix}
            1&0&0&0\\
            0&1&0&0\\
            0&0&-1&0\\
            0&0&0&-1\\
        \end{pmatrix}
        &(f_Y)_*=\begin{pmatrix}
            1&0&0&0\\
            0&-1&0&0\\
            0&0&1&0\\
            0&0&0&-1\\
        \end{pmatrix}\\
        &(f_Z)_*=\begin{pmatrix}
            1&0&0&0\\
            0&-1&0&0\\
            0&0&-1&0\\
            0&0&0&1\\
        \end{pmatrix}
        &(J)_*=\begin{pmatrix}
            -1&0&0&0\\
            0&1&0&0\\
            0&0&1&0\\
            0&0&0&1\\
        \end{pmatrix}
    \end{align*}

    To achieve this, simply attach each of the three copies $M_X,M_Y,M_Z$ at the three fixed points of the action, all by nontrivial hinge connect sums.
\end{constr}

\begin{figure}[hbt!] 
    \centering
    \includegraphics[width=0.7\textwidth]{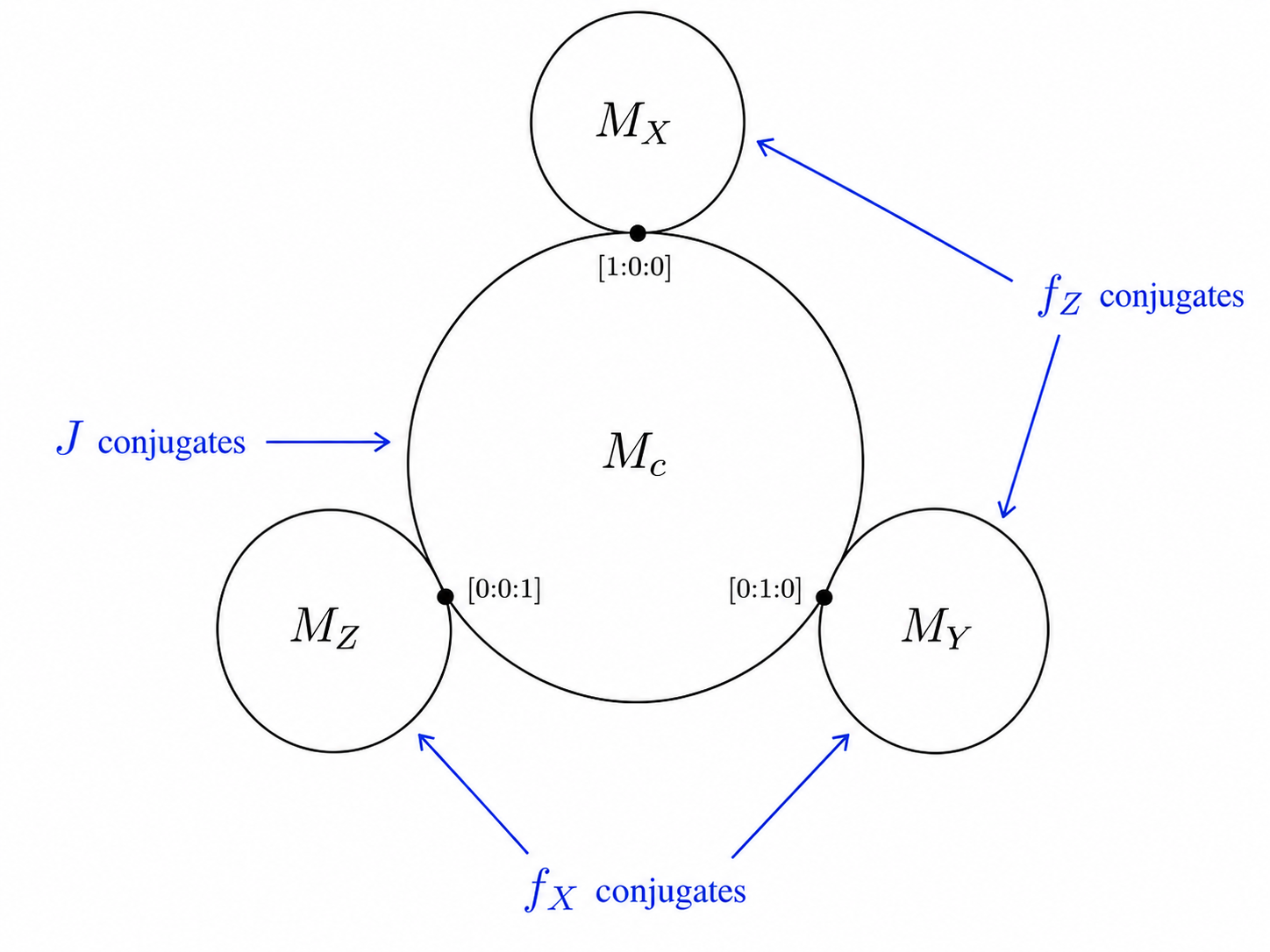}
    \caption{\Cref{constr:2}}
    \label{fig:3hinge} 
\end{figure}

\vspace{5mm}

\begin{constr}\label{constr:3}
    Extend the ideas of \Cref{constr:2} from a local to global model, as follows.  A copy of $\cp^2$, endowed with some projective coordinates $[X:Y:Z]$, comes equipped with three fixed points of the group of diffeomorphisms $\ang{f_X,f_Y,J}$ which are called its \textit{coordinate points}.  Given two copies of $\cp^2$, they can be connected to each other in a neighborhood of respective coordinate points, via either the hinged or standard connect sum from \Cref{constr:1}. By a $\cp^2$\textit{-tree}, denote a copy of $M_n$ given locally in projective coordinates on each copy of $\cp^2$, and where different copies are connected at respective coordinate points by hinge or standard connect sums in this way.  Such a $\cp^2$-tree gives an adjacency graph whose vertices are the copies of $\cp^2$ and where edges denote a local connect-sum.  (We can also choose to 2-color the edges to distinguish trivial and hinge connect sums.)  Any such $\cp^2$-tree of size $n$ yields an action of $(\Z/2\Z)^3\leq G_n\cong (\Z/2\Z)^n$ on $M_n$, induced by $\ang{f_X,f_Y,J}$ on any given copy of $\cp^2$.  A chain of $\cp^2$s is a $\cp^2$-tree whose associated adjacency graph is a path.

    \begin{figure}[hbt!] 
    \centering
    \includegraphics[width=0.7\textwidth]{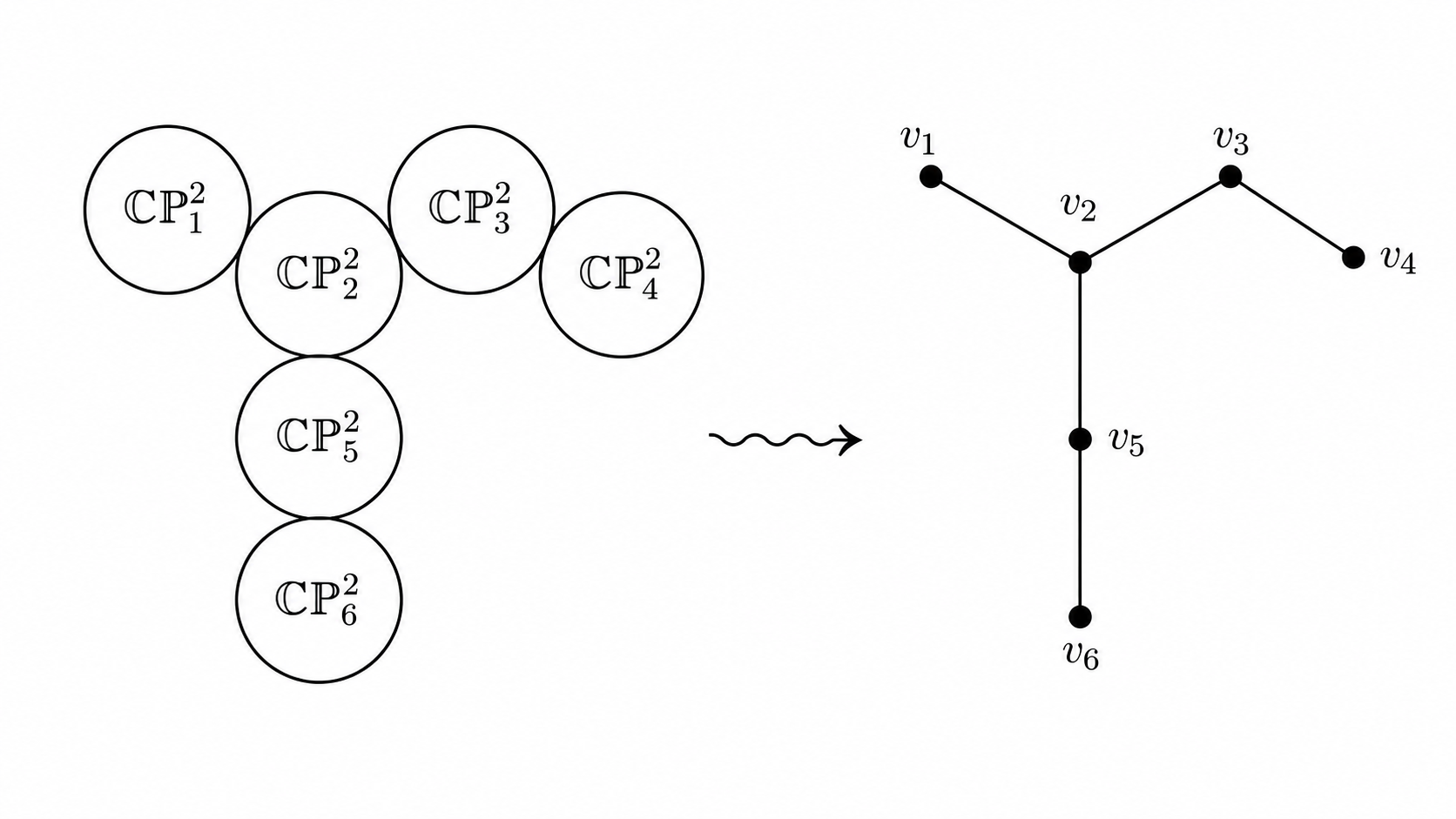}
    \caption{An example of \Cref{constr:3}, with its adjacency graph to the right. The graph does not depend on whether connect sums are trivial or hinged}
    \label{fig:graph-example} 
    \end{figure}
\end{constr}

\vspace{5mm}

\begin{eg}\label{eg:g3}
    The group $G_3\cong (\Z/2\Z)^3$ is realizable in $\diff^+(M_3)$.  To see this, consider a chain of three $\cp^2$s, with both connect sums hinged.  Then the group $\ang{f_X,f_Y,J}$ on any given copy induces precisely a realization of $G_3$ in $\diff^+(M_3)$.
\end{eg}

\vspace{5mm}

\begin{eg}\label{eg:g4}
    Call two subgroups of $G_n\cong (\Z/2\Z)^n\leq O(H_{M_n})$ \textit{permutation-equivalent} if there is an isomorphism between them induced by some permutation of the standard basis vectors $e_i,\dots,e_n$ of $H_2(M_n;\Z)$ (the hyperplane classes of each $\cp^2$).  Clearly, realizability is invariant under permutation-equivalence.
    
    Consider $G_4$ generated by reflections $e_1,e_2,e_3,e_4$ on the standard basis vectors of $H_2(M_n;\Z)$.  To apply \Cref{constr:3} in a nontrivial way (i.e., to actually realize a group of rank 3), our configuration graph of connect sums must contain a vertex of degree 3.  Then there is only one valid tree to consider up to tree isomorphism, the star graph.  By considering different choices of standard and hinged connections along each edge, \Cref{constr:3} realizes two different groups up to permutation-equivalence. By choosing all three edges hinged (as in \Cref{fig:3hinge}), realize the group $\ang{e_1,e_2e_3,e_3e_4}$ where $e_1$ denotes reflection on the central copy of $\cp^2$.  Choosing only two hinged connections and one standard, realizes the group $\ang{e_1,e_2,e_3e_4}$, where here $e_3$ and $e_4$ reflect on the two copies of $\cp^2$ which are linked by a standard connection.
    
    Similarly, consider $G_5$ with standard generators $e_1,\dots,e_5$.  There is similarly only one graph on five vertices with degree at most 3 and containing a vertex of degree 3.  (For example, begin with the star graph of size 4 and connect an additional vertex to one of its leaves.)  Then by direct computation, \Cref{constr:3} realizes precisely the following groups on $M_5$, up to permutation-equivalence:
    \[\ang{e_1,e_2,e_3e_4e_5},\ang{e_1,e_2e_3,e_4e_5},\ang{e_1,e_2e_3,e_3e_4e_5},\ang{e_1e_2,e_3e_4,e_4e_5}.\]
\end{eg}

\vspace{5mm}

One other application of \Cref{constr:3} is to prove the first part of \Cref{thm:large_rank}.
\begin{proof}[Proof of the first part of \Cref{thm:large_rank}]
    Realize arbitrary $H\leq G_n$ of rank 2 in $\diff^+(M_n)$ as follows.  Consider two generators of $H$, called $\phi$ and $\psi$.  Let $S_\phi\coloneqq\{e_1,\dots,e_n\}\cap V_{-1}(\phi)$, where $V_{-1}(\phi)$ denotes the $(-1)$-eigenspace of $\phi$.  In other words $S_\phi$ is the set of basis vectors reflected by $\phi$. Define $S_\psi$ similarly.  Now define sets \[A=S_\phi\setminus S_\psi, B=S_\psi\setminus S_\phi, C=S_\phi\cap S_\psi, D=(S_\phi\cup S_\psi)^c.\]
    
    First, assume that $D\neq\emptyset$.  Start with a chain of $\abs{D}$ $\cp^2$s, all attached trivially.  Then to a leaf $N$ of the chain, attach two more trivially-connected chains of lengths $\abs{A}$ and $\abs{B}$, each via a hinge sum.  Finally, to any available member of the original chain, attach with a hinge sum a trivially connected chain of length $\abs{C}$.  Then use $f_X$ and $f_Y$ on $N$ to realize $H$.

    In the case that $\abs{D}=0$ but $\abs{C}\neq0$, the construction is similar. Start with a chain of $\abs{C}$ $\cp^2$s using all trivial connect sums, and to a leaf $N$ attach trivially-connected chains of lengths $\abs{A}$ and $\abs{B}$, each via a hinge sum.  Then use $Jf_X$ and $Jf_Y$ on $N$ to realize $H$.

    If $\abs{C}=\abs{D}=0$, then simply begin instead with \Cref{constr:1}, and attach trivially connected chains trivially to either starting copy of $\cp^2$, of lengths $\abs{A}-1$ and $\abs{B}-1$, respectively.

    Of course, it cannot be that $\abs{C}=\abs{D}=\abs{A}=0$, since then $\phi$ would act trivially and hence is the identity in $H$.

\begin{figure}[hbt!] 
    \centering

\begin{tikzpicture}[
  line cap=round,
  line join=round,
  circ/.style={draw=black,line width=1.15pt},
  dot/.style={fill=black},
  cut/.style={red,line width=2.5pt},
  brace/.style={
    blue,
    line width=1.25pt,
    decorate,
    decoration={brace,amplitude=7pt}
  },
  lab/.style={blue,font=\fontsize{22}{24}\selectfont}
]

\def\R{0.32}
\pgfmathsetmacro{\sep}{2*\R}
\def\theta{38}        
\def\cutHalf{0.34}    
\def\dotR{0.035}

\coordinate (O) at (0,0);
\draw[circ] (O) circle[radius=\R];

\foreach \i in {1,...,8}{
  \coordinate (V\i) at (0,{\i*\sep});
  \draw[circ] (V\i) circle[radius=\R];
}

\foreach \t in {8.62,8.92,9.22}{
  \fill[dot] (0,{\t*\sep}) circle[radius=\dotR];
}
\coordinate (VT) at (0,{10.15*\sep});
\draw[circ] (VT) circle[radius=\R];

\foreach \i in {1,...,5}{
  \coordinate (L\i) at ({-\i*\sep*cos(\theta)},{-\i*\sep*sin(\theta)});
  \draw[circ] (L\i) circle[radius=\R];
}
\foreach \t in {5.62,5.92,6.22}{
  \fill[dot] ({-\t*\sep*cos(\theta)},{-\t*\sep*sin(\theta)})
    circle[radius=\dotR];
}
\foreach \t in {7.00,7.95}{
  \draw[circ] ({-\t*\sep*cos(\theta)},{-\t*\sep*sin(\theta)})
    circle[radius=\R];
}

\foreach \i in {1,...,5}{
  \coordinate (R\i) at ({\i*\sep*cos(\theta)},{-\i*\sep*sin(\theta)});
  \draw[circ] (R\i) circle[radius=\R];
}
\foreach \t in {5.62,5.92,6.22}{
  \fill[dot] ({\t*\sep*cos(\theta)},{-\t*\sep*sin(\theta)})
    circle[radius=\dotR];
}
\foreach \t in {7.00,7.95}{
  \draw[circ] ({\t*\sep*cos(\theta)},{-\t*\sep*sin(\theta)})
    circle[radius=\R];
}


\coordinate (MV) at ($(V6)!0.5!(V7)$);
\draw[cut] ($(MV)+(-0.43,0)$) -- ($(MV)+(0.43,0)$);

\coordinate (ML) at ($(O)!0.5!(L1)$);
\draw[cut]
  ($(ML)+({\theta-90}:\cutHalf)$) --
  ($(ML)+({\theta+90}:\cutHalf)$);

\coordinate (MR) at ($(O)!0.5!(R1)$);
\draw[cut]
  ($(MR)+({90-\theta}:\cutHalf)$) --
  ($(MR)+({270-\theta}:\cutHalf)$);


\draw[brace,decoration={brace,amplitude=7pt,mirror}]
  ($({-7.00*\sep*cos(\theta)},{-11.00*\sep*sin(\theta)})+({90+\theta}:0.56)$) -- ($({0.5*\sep*cos(\theta)},{-3*\sep*sin(\theta)})+({90+\theta}:0.56)$);
\node[lab] at
  ($({-1.5*\sep*cos(\theta)},{-9*\sep*sin(\theta)})+({90+\theta}:1.18)$)
  {$A$};

\draw[brace]
  ($({0.72*\sep*cos(\theta)},{-0.72*\sep*sin(\theta)})+({90-\theta}:0.56)$) --
  ($({8.5*\sep*cos(\theta)},{-8.5*\sep*sin(\theta)})+({90-\theta}:0.56)$);
\node[lab] at
  ($({4.5*\sep*cos(\theta)},{-4*\sep*sin(\theta)})+({90-\theta}:1.18)$)
  {$B$};

\draw[brace]
  (-0.6,{6.55*\sep}) -- (-0.6,{10.52*\sep});
\node[lab] at (-1.62,{8.58*\sep}) {$C$};

\draw[brace]
  (-0.6,{0.05*\sep}) -- (-0.6,{6.45*\sep});
\node[lab] at (-1.62,{3.5*\sep}) {$D$};

\end{tikzpicture}
    
    \caption{Realization in rank 2, as in the proof of the first part of \Cref{thm:large_rank}.  Hinge connect sums are marked in red}
    \label{fig:rk2} 
\end{figure}
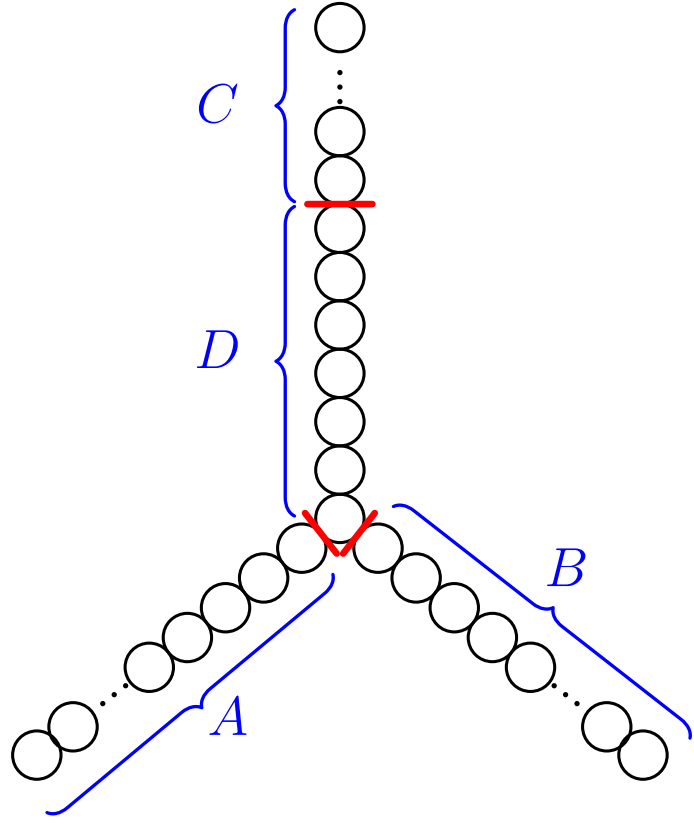
\end{proof}

\begin{rmk}
    No subgroups $H\leq G_n\leq O(H_{M_n})$ of rank at least $4$ can be realized in $\diff^+(M_n)$ via equivariant connect sums of $\cp^2$.  For example, Wilczynski and Hambleton-Lee proved that any finite group acting effectively by orientation-preserving diffeomorphisms on $\cp^2$ and trivially on its homology must embed into $PGL(3,\C)$ (for example, see \cite{mmbz}).  Since the $(\Z/2\Z)^k$ embeds into $PGL(3,\C)$ only if $k\leq 2$ and since $\aut(H_2(\cp^2,\Z);\Z)\cong\Z/2\Z$, then $(\Z/2\Z)^4$ does not act effectively on $\cp^2$.
\end{rmk}

\subsection{Tools for obstructing high rank realization}\label{sec:tools}

We introduce a few tools for obstructing realization in $G_n$.  Let $G=\Z/p\Z$ act by orientation-preserving diffeomorphisms on a smooth, closed, oriented 4-manifold $M$.  Its fixed set $M^G\subseteq M$ is then a disjoint union of isolated points and embedded surfaces (see e.g. \cite{blk3}, Proof of Lemma 3.5 (3)).

\subsubsection{Betti numbers of $M^G$}

The action of $G$ on $M$ induces an action of $G$ on $H_2(M;\Z)$.  By \cite{ed}, $H_2(M;\Z)$ decomposes into a direct sum of indecomposable representations of three types: trivial of dimension one, cyclotomic of dimension $p-1$, and regular of dimension $p$.  Note that the representations of cyclotomic type (resp. regular type) need not be isomorphic as $G$-representations to $\Z[\zeta_p]$ for $\zeta_p$ a primitive $p$th root of unity (resp. to $\Z[G]$), when $p\geq 23$.  However, this paper only uses small primes $p<23$ in all applications.  Notice that in dimension 2, a trivial representation is given by the matrix $(1)$, a cyclotomic representation is given by $(-1)$, and a regular representation is given by the matrix \[\sigma=\begin{pmatrix} 0&1\\1&0 \end{pmatrix}.\]

\begin{prop}[\cite{ed}, Proposition 2.4]\label{prop:ed}
    Let $G=\Z/p\Z$ act by orientation-preserving diffeomorphisms on a smooth, closed, oriented 4-manifold $M$.  Let $t$ be the number of trivial summands and $c$ the number of cyclotomic summands in $H_2(M;\Z)$.  
    \begin{enumerate}
        \item If $M^G\neq\emptyset$ then
        \begin{enumerate}
            \item $\beta_1(M^G)=c$ \text{ and}
            \item $\beta_0(M^G)+\beta_2(M^G)=t+2$,
        \end{enumerate}
        where $\beta_k(M^G)$ denotes the $k$th mod-$p$ Betti number of $M^G$.

        \item The Euler characteristic $\chi(M^G)=t-c+2$.  In particular, if $t-c+2\neq0$, then $M^G\neq\emptyset$ and part (1) applies. 
    \end{enumerate}    
\end{prop}

\subsubsection{The Hirzebruch $G$-signature theorem}
Again, let $M$ be a smooth, closed, oriented 4-manifold on which $G=\Z/p\Z$ acts by orientation-preserving diffeomorphisms.  At an isolated point $p\in M^G$, the $G$-action on $T_pM$ preserves a complex structure compatible with the ambient orientation and hence is described by a 2-dimensional, reducible complex representation of the form \[g\in G:(z_1,z_2)\mapsto (\chi_1(g)z_1,\chi_2(g)z_2)\] for two nontrivial complex characters $\chi_1,\chi_2$ of $G$.  However, presentation is not unique, since the conjugate complex structure also has this property (and no other complex structures), for which the characters are $(\bar{\chi_1},\bar{\chi_2})$.  Then define the \textit{signature defect} at $z$ by \[\de_z\coloneqq\sum_{g\in G\setminus\{1\}}\frac{(1+\chi_1(g))(1+\chi_2(g))}{{(1-\chi_1(g))(1-\chi_2(g))}}=\sum_{\zeta\in\mu_p\setminus\{1\}}\frac{(1+\zeta)(1+\zeta^q)}{{(1-\zeta)(1-\zeta^q)}}.\]  The second equality follows because $G$ has prime order, and so $\chi_2=\chi_1^q$ for some $q\in\F_p^*$, and both characters are generators of the group of characters of $G$.

For each 2-dimensional path component $C\subset M^G$, let \[\de_C\coloneqq\frac{p^2-1}{3}\ (C\cdot C),\] where $C\cdot C$ denotes the self-intersection of $C$ with itself.  When $C$ is orientable, this is simply equal to the standard intersection product $[C]\cdot [C]$ for $[C]\in H_2(M,\Z)$ the homology class of $C$ as an embedded surface.  In the non-orientable case, there is still a sensible self-intersection number as follows.  Isotope a second copy of $C$ in $M$ so that it is transverse to its original position; denote this isotoped surface by $C'$.  Now consider an intersection point $z\in C\cap C'$, and note that $T_zM=T_zC\oplus T_z C'$.  A choice of basis $(z_1,z_2)$ for $T_zC$ induced a basis $(z_1',z_2')$ for $T_zC'$ since $C'$ was produced by an isotopy of $C$.  Hence define the signed self-intersection of $C$ at $z$ to be the orientation of the basis $(z_1,z_2,z_1',z_2')$, relative to the ambient orientation on $M$.  By construction, this sign is invariant both under a change of basis for $C$ and under a reordering of $C$ and $C'$. (Clearly, this construction reduces to the standard intersection sign in the orientable case.)  Define $C\cdot C$ to the sum of local self-intersection numbers.

Since $M/G$ is a rational homology manifold, it has a well-defined signature $\sigma(M/G)$.  In practice, we identify this signature with the signature of the restricted intersection pairing on $H_2(M;\R)^G$.  We can now state the main theorem.
\begin{theorem}[Hirzebruch $G$-Signature Theorem, see e.g. \cite{hirze}, (12) on page 177]\label{thm:gsig}
    Let $M$ be a closed, oriented 4-manifold on which $G=\Z/p\Z$ acts by orientation-preserving diffeomorphisms.  Then \[p\cdot\sigma(M/G)-\sigma(M)=\sum_z\de_z + \sum_C\de_C,\] where the sums range over all 0- and 2-dimensional path components of $M^G$, respectively.
\end{theorem}

\subsubsection{Finite group actions on surfaces}
In particular, the proof of \Cref{thm:large_rank} eventually reduces to finite group actions on surfaces. For a fuller exposition of the theory of finite groups acting on surfaces, see \cite{primer}.  Recall that if a finite group $G$ acts on the genus-$g$ surface $S_g$, then the action can be isotoped so that $G$ acts by isometries in some hyperbolic metric. Then the following two theorems control the size of groups that act faithfully on $S_g$.

\begin{theorem}\label{thm:84(g-1)}
    If $X$ is a closed hyperbolic surface of genus $g\geq 2$, then \[\abs{\isom^+(X)}\leq 84(g-1).\]
\end{theorem}

\begin{theorem}\label{thm:4g+2}
    Let $X$ be a closed hyperbolic surface of genus $g\geq 2$.  Then any element of $\isom^+(X)$ has order at most $4g+2$.
\end{theorem}

When $g=1$, recall that if $(\Z/2\Z)^k$ acts faithfully on the torus $T^2$, then $k\leq 4$.  For $g=0$, uniformization gives that if $(\Z/2\Z)^k$ acts faithfully on the sphere $S^2$, then $k\leq 2$.

Recall that the non-orientable surface $N_g$ of (non-orientable) genus $g>0$ is defined to be the connect sum of $g$ copies of $\rp^2$.  Note that $\chi(N_g)=2-g$, and hence its orientation double cover is $S_{g-1}$. As such, any group acting faithfully on $N_g$ has an index 2 subgroup which acts faithfully on $S_{g-1}$.

\subsection{Proving \Cref{thm:large_rank} and \Cref{cor:gn}}\label{sec:pfthmlgrk}

The proof of the second part of \Cref{thm:large_rank} requires several steps, but the core strategy is the following: Suppose $H\leq G_n$ is realized in $\diff^+(M_n)$.  First, show that there is some element $\phi\in H$ whose fixed set includes either an isolated point or a surface of low genus.  Second, analyze the action of $H$ on the fixed set of $\phi$ in order to bound the size of $H$, noting that $H$ must act faithfully on $M_n$.\\

\noindent To start, note the following obstruction to free actions in $G_n$.

\begin{lem}\label{lem:no-free-involution}
    Let $n\geq 1$. No order 2 $\phi\in O(H_{M_n})$ can be realized by a smooth involution in $\diff^+(M_n)$ which acts freely on $M_n$.
\end{lem}
\begin{proof}
    Suppose $\phi\in O(H_{M_n})$ does have a free-acting representative $f\in\diff^+(M_n)$.  (In particular, $f_*=\phi$.)  In the language of \Cref{prop:ed}, let $t,c$, and $r$ refer to the number of trivial, cyclotomic, and regular summands of $\phi$. Then $t+2-c=0$, and $t+c+2r=n$.  The quotient manifold has rational homology \[H_2(M_n/f;\Q)\cong H_2(M_n;\Q)^{\phi},\] by transfer.  Then the rank of $H_2(M_n/f;\Q)=t+r$.  Since signature is multiplicative in covers, $\sigma(M_n/f)=\sigma(M_n)/2=n/2$. Then \[n=t+c+2r=t+(t+2)+2r=2(t+r+1),\] which contradicts $t+r=n/2$.  So $\phi$ has no free-acting involution representative.
\end{proof}

\noindent We also record the following combinatorial lemma. Recall that $\ell(\phi)$ for $\phi\in G_n\leq O(H_{M_n})$ denotes the dimension of its $(-1)$-eigenspace $V_{-1}(\phi)\leq H_2(M_n;\Q)$.
\begin{lem}\label{lem:2/3}
    Let $n\geq 1$. When $\rk(H)>1$, there is some $\phi\in H$ such that $\ell(\phi)\leq 2n/3$.
\end{lem}
\begin{proof}
     Note that $\ell(\phi)=n-\dim(H_2(M_n;\Q)^{\ang{\phi}})$. Suppose that there are two nontrivial elements $1\neq \phi,\psi\in H$ such that $\ell(\phi),\ell(\psi)>2n/3$.  Then $\ell(\phi\psi)=\ell(\phi)+\ell(\psi)-2t$, where $t=\dim (V_{-1}(\phi)\cap V_{-1}(\psi))$.  Also, $t\geq \ell(\phi)+\ell(\psi)-n$ by a pigeon-hole argument, and so \begin{align*}
        \ell(\phi\psi)&\leq\ell(\phi)+\ell(\psi)-2(\ell(\phi)+\ell(\psi)-n)\\
        &=2n-(\ell(\phi)+\ell(\psi))\\
        &<2n-4n/3\\
        &=2n/3.
    \end{align*}
\end{proof}

\noindent The next lemma serves to convert fixed-set data into a bound on the size of a realizable subgroup $H\leq G_n$.
\begin{lem}\label{lem:rankorbitsize}
    Let $n\geq 1$, and let $f\in \diff^+(M_n)$ be a diffeomorphism realizing some $\phi\in G_n\leq O(H_{M_n})$.  Fix some surface type appearing in $Fix(f)$, and suppose $Fix(f)$ contains $k$ disjoint surfaces of that type.  Then there exists a surface $S\subseteq Fix(f)$ of the given type such that $\abs{Stab_S(\tilde H)}\geq2^{\max\{0,\rk(H)-v_2(k)\}}$, where $v_2(k)$ denotes 2-adic valuation.
\end{lem}
\begin{proof}[Proof of \Cref{lem:rankorbitsize}]
    Suppose $H$ is realized by some $\tilde H\leq \diff^+(M_n)$. Consider the set $A_S$ of all surfaces of the given type in $Fix(f)$.  Note that $\tilde H$ must act on $A_S$ with $f$ acting by the identity.  We show that the action of $\tilde H$ on $A_S$ has an orbit of size at most $2^{\min\{\rk(H),v_2(k)\}}$, and then the claim follows by the orbit-stabilizer theorem.

    Let $\{C_i\}$ be the orbit sizes, and let $C=\min\{C_i\}$ be the size of the smallest.  Since $H$ is a 2-group, $C$ must be a power of 2 which is smaller than (or equal to) the size of all other orbits.  Then $C|C_i$ for all $i$. Then $C|\sum_iC_i=k$. Since $C$ is a power of 2 dividing $k$, and $C$ divides $\abs{H}$ by the orbit-stabilizer theorem, then $C=2^r$ for some $r\leq \rk(H)$ and $r\leq v_2(k)$.  The result follows.
\end{proof}

\noindent The next proposition provides the main input needed to prove \Cref{thm:large_rank}.  It first shows that if $H\leq G_n$ is realizable in $\diff^+(M_n)$, some element $\phi\in H$ must be realized by a diffeomorphism whose fixed set contains a connected component of uniformly bounded complexity.  Then the theory of group actions on surfaces bounds the size of this realizable $H$ in terms of the fixed-set data of this element $\phi$.  Some of the proof of the following lemma was inspired by \cite{bar}, Section 6.
\begin{prop}\label{prop:fullthm}
    Let $n\geq 1$, and let $H\leq G_n$ with $\rk(H)=m$.  Assume $H$ is realizable in $\diff^+(M_n)$ as some group $\tilde H<\diff^+(M_n).$
    \begin{enumerate}
        \item There is some $\phi \in H$ such that the fixed set $Fix(f)$ of a diffeomorphism $f$ realizing $\phi$ contains either an isolated point, $S^2, T^2, S_2, \rp^2, N_2, N_3,$ or $N_4$, where $N_g$ denotes the non-orientable surface of genus $g$.
        \item If the fixed set $Fix(f)$ contains $k$ isolated points, then $m<4+v_2(k)$, where $v_2$ denotes $2$-adic valuation.  If the fixed set $Fix(f)$ contains $k$ components all diffeomorphic to the same one of $S^2, T^2, S_2$, or $N_g$ for some $g\leq 4$, then $m<c+v_2(k)$, where $c$ is the largest integer such that $(\Z/2\Z)^{c-1}$ acts faithfully on the given surface type.  (Note that $c$ is some finite constant depending on the surface type.  In particular, $c<8$ in all cases.)
    \end{enumerate}
\end{prop}

\begin{proof}
    We first prove part 1 of the statement.  By \Cref{lem:2/3}, let $\phi\in H$ with $\ell(\phi)\leq 2n/3$ in the language of the proof of \Cref{lem:2/3}, and suppose $f$ is a smooth involution realizing $\phi$.  If $Fix(f)$ contains isolated points, then part 1 is satisfied immediately, so assume it does not.  Rephrase the proposition statement as follows: $Fix(f)$ contains a compact surface $S$ with $\beta_1(S)\leq4$, where $\beta_1$ denotes the first Betti number with $\F_2$ coefficients.  By the classification of surfaces, this is equivalent to the original statement.
    
    Since $\ell(\phi)\leq 2n/3$, then for $\phi$, $c\leq 2t$, where $c,t$ are in the language of Edmonds \Cref{prop:ed}.  In particular, each surface contributes 2 to $\beta_0(Fix(f))+\beta_2(Fix(f))$.  Since $\beta_0(Fix(f))+\beta_2(Fix(f))=t+2$, then there must be $(t+2)/2$ surface components.  (In particular, if $t$ is odd, we must be in the isolated fixed point case.)  Then by a pigeon-hole argument, one of these surface components $S$ must have $\beta_1(S)\leq 4$, since $c=\beta_1(Fix(f))$.

    We now prove the second part of the proposition. The key observation is that for each fixed surface $S$ in $Fix(f)$, the stabilizer $Stab_S(\tilde H)$ which fixes $S$ set-wise, must act almost faithfully on $S$, with the kernel of the action on $S$ being only $\ang{f}$ itself.  To see this, suppose there is $f\neq g\in Stab_S(\tilde H)$ which fixes $S$ pointwise.  Since the order $\abs{g}=2$, then $g$ must act by $-1$ on the normal bundle of $S$.  Of course $f$ also fixes $S$ pointwise.  Then $fg$ acts trivially on $S$ and its normal bundle.  Since any isometry (and hence any finite order diffeomorphism by averaging a metric) is determined by its action on a point and a frame, $fg$ must be the identity.  But since $f\inv=f\neq g$, there is a contradiction.  Hence $Stab_S(\tilde H)/\ang{f}$ acts faithfully on $S$.  Then \Cref{lem:rankorbitsize} implies the sought statement $m<c+v_2(k)$.
    
    All that remains is the universal bound on the constant $c$.  Any diffeomorphism of a non-orientable surface $N_g$ lifts to a diffeomorphism of its orientation double-cover $S_{g-1}$.  On $S^2$ and $T^2$ (and by extension $\rp^2$ and $N_2$ by passing to the orientation cover), uniformization gives the desired bound.  Then apply the $84(g-1)$ theorem for hyperbolic surfaces on the other surface types.
\end{proof}

Now the proof of \Cref{thm:large_rank} follows quickly from \Cref{prop:fullthm}.

\begin{proof}[Proof of \Cref{thm:large_rank}, part 2]
    We need to show that each of the following three conditions in the theorem statement prohibits realization in $\diff^+(M_n)$: 
    \begin{enumerate}
        \item $Rk(H) \geq 4$ and $n$ is odd
        \item $Rk(H) \geq 4$ and there is some $\phi\in H$ with $(n-\ell(\phi))\equiv 1\mod 2$, where $\ell(\phi)$ denotes the number of $(-1)$s in its matrix representation.
        \item $Rk(H) > 8+\log_2(n)$.
    \end{enumerate}
    
    Obstruction given the third condition follows directly from \Cref{prop:fullthm}, by taking a universal bound over all surface types, and noting that in the language of that proposition, $v_2(k)\leq\log_2(k)\leq\log_2(n)$.

    Obstruction in the second case follows by noting that this condition ensures that $\beta_0(Fix(f))+\beta_2(Fix(f))=n-\ell(\phi)+2$ is odd, and hence there must be an odd number of isolated fixed points.  Then by orbit-stabilizer, there must be some global fixed point of $\tilde H$ since $\tilde H$ is a 2-group.  Finally, note that $(\Z/2\Z)^m\into SO(4)$ precisely for $m\leq 3$.

    For the first condition, reduce to the second.  When $n$ is odd, then the second condition asks for $\phi\in H$ with even length.  But notice that if $\phi,\psi\in H$ both have odd-length, then $\phi\psi$ has even length.  So for odd $n$, every subgroup of $G_n$ of rank at least 2 has an even-length element.
\end{proof}

\begin{proof}[Proof of \Cref{cor:gn}]
    By the first condition in part 2 of \Cref{thm:large_rank}, all three statements are trivial for odd $n$.  Assume then that $n$ is even.
    
    Let $a_n$ denote the probability that a random rank-$k$ subgroup of $G_n$ is realizable in $\diff^+(M_n)$.  We show that $\limsup_n a_n\leq 2^{-k}$, where $a_n$ denotes the realizable proportion of rank-$k$ subgroups of $G_n$.  By the second condition of 2 of \Cref{thm:large_rank}, any realizable subgroup of rank $k$ must have all elements of even length; call such a subgroup even. Note that the number of even subspaces of $\F_2^n$ of dimension $k$ is equal to the total number of subspaces of $\F_2^{n-1}$ of dimension $k$, since the even elements of $\F_2^n$ are precisely the kernel of the linear map which sends an element to its mod-2 Hamming weight.  Then the even proportion $P_e(n)$ of rank-$k$ subgroups is \[P_e(n,k)=\frac{\binom{n-1}{k}_2}{\binom{n}{k}_2}=\frac{2^{n-k}-1}{2^n-1}\sim2^{-k}.\]  Then since $a_n\leq P_e(n)$, then $\limsup_n a_n\leq 2^{-k}$.

    Let $b_n$ denote the probability that a random subgroup of $G_n$ is realizable in $\diff^+(M_n)$.  To show $\lim_n b_n=0$, apply the third condition of part 2 of \Cref{thm:large_rank} since subspaces of $\F_2^n$ are almost never of rank at most $8+\log_2n$.

    To show the third part, let $P_e(n)$ denote the even proportion of subgroups of $G_n$ whose rank is at most $N\coloneqq \lfloor \log_2 n\rfloor$. By the logic of part a, it suffices to show that \[P_e(n)=\frac{\sum_{k=0}^{N}\binom{n-1}{k}_2}{\sum_{k=0}^{N}\binom{n}{k}_2}\xrightarrow[n\to \infty]{} 0.\]  And this is true because to first order each sum is simply given by its dominant term, and \[\frac{\binom{n-1}{N}_2}{\binom{n}{N}_2}=\frac{2^{n-N}-1}{2^{n}-1}\sim 2^{-N}=n\inv.\]  Then $\lim_{n\to\infty} c_n\leq \lim_{n\to\infty}P_e(n)=0$, where $c_n$ denotes the probability that a random subgroup $H\leq G_n$ with rank at most $\log_2n$ is realizable in $\diff^+(M_n)$.
\end{proof}

\section{Realization in $S_n<O(H_{M_n})$}\label{sec:sn}

This section analyzes realizability inside the permutation subgroup $S_n<O(H_{M_n})$; to do so, we rely strongly on a result of Hambleton--Tanase \cite{ht}.  We explain the result fully below, but informally, it says:

\begin{adjustwidth}{2em}{0pt}
Suppose $\sigma\in S_n<O(H_{M_n})$ has odd order, and suppose that $\sigma$ is realizable in $\diff^+(M_n)$.  Then $\sigma$ can be realized using an equivariant connected sum of $1\leq j\leq n$ copies of $\cp^2$, on each of which the action is linear (i.e., a diagonal matrix action on standard projective coordinates of each copy).
\end{adjustwidth}

\noindent Hambleton--Tanase package the realizability question into a combinatorial question. On the one hand, their theorem specifies exactly which odd order permutations are realizable: one enumerates all possible relevant combinatorial objects for all permutations, and then checks if a specific permutation has combinatorial data in that list.  On the other hand, their formulation does not give an easily computable criterion to determine realizability of a large permutation. The first goal of this section is to prove the following theorem, which implies a piece of part 1 of \Cref{thm:mega-thm}.

\begin{theorem}\label{thm:asym}{\bf (Asymptotic non-realizability in $S_n< O(H_{M_n})$: cyclic groups)}\hfill
    \begin{enumerate}
        \item Random elements of odd order in $S_n<O(H_{M_n})$ are asymptotically almost never realizable in $\diff^+(M_n)$.
        
        \item Let $t:\mathbb N\to\mathbb Z_{\ge0}$ satisfy
        \[t(n)=o\left(\frac{\sqrt n}{\log n}\right).\]
        Up to conjugacy, random odd order elements of $S_n<O(H_{M_n})$
        with at most $t(n)$ fixed points as permutations are asymptotically almost never realizable in
        $\mathrm{Diff}^+(M_n)$.
    \end{enumerate}
\end{theorem}

\noindent We prove \Cref{thm:asym} in \Cref{sec:sn} via results on partition and permutation statistics.  The proofs follow by relatively standard methods in those fields, but it is notable that analytic number theory or probabilistic arguments would be relevant to a Nielsen realization problem. The bound in \Cref{thm:asym} (2) only uses some small portion of the available information, albeit some of the most accessible information.  In fact, we suspect that very few permutations are realizable for large $n$, even up to conjugacy.

\begin{conj}\label{conj}
    Up to conjugacy, random elements of odd order in $S_n<O(H_{M_n})$ are asymptotically almost never realizable in $\diff^+(M_n)$.
\end{conj}

\noindent To pair with \Cref{thm:asym} about odd order cyclic groups of $S_n<O(H_{M_n})$, we prove the following analog about $2$-groups in $S_n<O(H_{M_n})$.

\begin{theorem}\label{thm:2groups}{\bf (Asymptotic non-realizability for $S_n<O(H_{M_n})$: $2$-groups)}\hfill
    \begin{enumerate}
        \item  Random $2$-subgroups of $S_n< O(H_{M_n})$ are asymptotically almost never realizable in $\diff^+(M_n)$.
        \item Up to conjugacy, random $2$-subgroups of $S_n< O(H_{M_n})$ are asymptotically almost never realizable in $\diff^+(M_n)$.
    \end{enumerate}
    In both settings, the asymptotic realizable probability is $O(e^{-cn^2})$ for some constant $c>0$.
\end{theorem}

\noindent \Cref{prop:polygon} below presents some simple constructions based on regular polytopes for realizing certain other subgroups of $S_n<O(H_{M_n})$.  \Cref{cor:Sn-to-8} demonstrates certain non-realizable elements of $S_n\leq O(H_{M_n})$, in particular at least one for each $n\geq 15$.  Finally, \Cref{prop:sn-obstr} below obstructs a certain class of large subgroups of $S_n<O(H_{M_n})$, defined in its statement. Together, these statements imply the following.

\begin{prop}\label{cor:Sn-summary}
    For $n\leq 8$, all cyclic subgroups of $S_n<O(H_{M_n})$ are realizable in $\diff^+(M_n)$.  In contrast, for any $n\geq 15$ there exists a cyclic subgroup of $S_n$ not realizable in $\diff^+(M_n)$.  When $n\geq 8$, the group $S_n<O(H_{M_n})$ itself is not realizable in $\diff^+(M_n)$.
\end{prop}

In \Cref{sec:cyclic}, we briefly present some work of Hambleton--Tanase \cite{ht}. We then state two propositions and use them to prove \Cref{thm:asym}, before proving these two propositions.  In \Cref{sec:2group}, we prove \Cref{thm:2groups}.  Finally in \Cref{sec:largeSn}, we present a set of related constructions realizing certain non-cyclic subgroups of $S_n<O(H_{M_n})$, together with an obstruction for many large non-cyclic subgroups.

\subsection{Odd order elements}\label{sec:cyclic}

We now introduce some results of Hambleton--Tanase \cite{ht}, in which the authors examine faithful $\Z/m\Z$-actions on $M_n$, for odd $m$.  Some of the nomenclature is our own, in order to exposit only what is needed here.  Let $t$ denote a generator of $\Z/m\Z$, and let $\zeta_m$ be a primitive $m$th root of unity.  On a fixed copy of $\cp^2$ with projective coordinates $[X:Y:Z]$, an action of $\Z/m\Z$ is called a \textit{standard linear action} if it is of the form \[t\mapsto \begin{pmatrix}
    1&&\\&\zeta^a&\\&&\zeta^b
\end{pmatrix}\in PGL_3\C< \diff^+(\cp^2)\] in projective coordinates, for some fixed $a,b\in\Z$ with $(a,b,m)=1$.  Note that any standard linear action of $t$ acts trivially on $H_2(\cp^2;\Z)$ and fixes the three coordinate points $p_X=[1:0:0]$, $p_Y=[0:1:0]$, and $p_Z=[0:0:1]$.  At $p_X$, consider the affine chart $\{[1:Y:Z]: Y,Z\in\C\}$, identified with $\C^2$ by $[1:Y:Z]\mapsto (Y,Z)\in\C^2$.  Identifying $\C^2$ with the complex tangent space $T_{(0,0)}\C^2$ in the usual way, observe that $t$ acts on $T_{p_X}\cp^2$ by the matrix
\[dt_{p_X}=\begin{pmatrix}\zeta^a&\\&\zeta^b\end{pmatrix}\in GL(T_{p_X}\cp^2).\]
Similarly using affine coordinate charts $\{[X:1:Z]:X,Z\in\C\}$ at $p_Y$ and $\{[X:Y:1]:X,Y\in\C\}$ at $p_Z$ gives \[dt_{p_Y}=\begin{pmatrix}\zeta^{-a}&\\&\zeta^{b-a}\end{pmatrix}, dt_{p_Z}=\begin{pmatrix}\zeta^{-b}&\\&\zeta^{a-b}\end{pmatrix}.\]
To record this data, we say that the points $p_X,p_Y,p_Z$ have the \textit{rotation numbers} $(a,b),\ (-a,b-a)$, and $(-b,a-b)$, respectively.  Sometimes, for clarity, we write $(a,b;m)$ for $(a,b)$.  Note that these rotation numbers are well-defined modulo $m$ up to the relations $(a,b)\equiv (b,a)\equiv(-a,-b)$.  Consider a point $p\in\cp^2$; a calculation shows that the orbit of $p$ under a standard linear action as above has cardinality in the set $\{1,m,\frac{m}{(b,m)},\frac{m}{(a,m)},\frac{m}{(a-b,m)}\}$.\\

\noindent Now use standard linear actions to construct  $M_n$ with a $(\Z/m\Z)$ action as follows. First, consider two copies of $\cp^2$, each with a standard linear action such that one has rotation number $(a,b)$ at some coordinate point $p_1$, and the other has rotation number $(-a,b)$ at some coordinate point $p_2$.  Then there is a natural $(\Z/m\Z)$-equivariant connect sum $\cp^2\#\cp^2$ formed locally in neighborhoods of $p_1$ and $p_2$.  Note that $\Z/m\Z$ acts trivially on $H_2(\cp^2\#\cp^2;\Z)$.  More generally, form $M_k$ as a connect sum of $k$ copies of $\cp^2$ with standard linear $\Z/m\Z$-actions, where each connect sum is formed at coordinate points in distinct copies of $\cp^2$ with respective rotation numbers of the form $(a,b)$ and $(-a,b)$.

Note that the action of $\Z/m\Z$ on $H_2(M_k;\Z)$ is trivial. These $k$ copies of $\cp^2$ are called \textit{central}, and this $M_k$ is called the \textit{central core}.  Now let $p\in M_k$ for $k>0$ with $\Z/m\Z$-orbit $O_p$ of size $r$.  Then form the manifold $M_{k+r}$ as a $(\Z/m\Z)$-equivariant connect sum by gluing one copy of $\cp^2$ to $M_k$ at each $p'\in O_p$. These $r$ copies of $\cp^2$ are called \textit{peripheral}; their $r$ hyperplane classes in $H_2(M_{k+r};\Z)$ now form a single $\Z/m\Z$-orbit, and they are permuted cyclically by the action of $t$. (See \Cref{fig:cyclic}.)  Inductively, we may attach equivariantly additional peripheral copies of $\cp^2$ onto the existing one.  For more details, see the notion of \textit{admissible trees} \cite{ht}.

Alternatively when $k=0$, glue copies of $\cp^2$ equivariantly to a fixed orthogonal $\Z/m\Z$ action on $S^4$, which can always be uniformized to take the form of rational rotations in each of two orthogonal 2-planes.  Any $M_n$ formed in this way from central and peripheral copies of $\cp^2$ is said to carry a \textit{standard-linear action} (or for short, $M_n$ is standard-linear).  In either case, this construction arranges the copies of $\cp^2$ into a tree $T$, and let $T_0$ denote the tree of central copies.

In particular, notice that any permutation action on $H_2(M_n;\Z)$ with at most two non-trivial cycle lengths is realizable by a standard-linear action.  Furthermore, any two elements $\sigma,\tau\in S_n\leq O(H_{M_n})$ with inequivalent cycle types are non-conjugate in $O(H_{M_n})$.  The number of such permutations is at least $c_1n^2(\log n)^2$ for some constant $c_1$, and so at least as many pair-wise non-conjugate elements of $O(H_{M_n})$ are realizable in $\diff^+(M_n).$\\

\begin{figure}[hbt!] 
    \centering

\usetikzlibrary{arrows.meta}
\definecolor{cpblue}{RGB}{0,20,255}
\definecolor{cared}{RGB}{255,20,20}

\resizebox{0.75\textwidth}{!}{%
\begin{tikzpicture}[
  x=1pt,
  y=-1pt,
  line cap=round,
  line join=round,
  every node/.style={inner sep=0pt, outer sep=0pt},
  maincircle/.style={draw=black, line width=2.35pt, fill=white},
  mainline/.style={draw=black, line width=2.35pt},
  blueline/.style={draw=cpblue, line width=2.45pt},
  bluecircle/.style={draw=cpblue, line width=2.45pt, fill=white},
  redarrow/.style={
    draw=cared,
    line width=2.55pt,
    -{Latex[length=13pt,width=9pt]}
  },
  reddouble/.style={
    draw=cared,
    line width=2.55pt,
    {Latex[length=13pt,width=9pt]}-{Latex[length=13pt,width=9pt]}
  },
  bluearrow/.style={
    draw=cpblue,
    line width=2.55pt,
    -{Latex[length=13pt,width=9pt]}
  },
  blackarrow/.style={
    draw=black,
    line width=2.45pt,
    -{Latex[length=13pt,width=9pt]}
  },
  point/.style={circle, fill=black, minimum size=13.5pt, inner sep=0pt},
  plabel/.style={font=\fontsize{39}{42}\selectfont, text=black},
  biglabel/.style={font=\fontsize{43}{48}\selectfont},
  bluelabel/.style={font=\fontsize{43}{48}\selectfont, text=cpblue}
]

\path[use as bounding box, fill=white] (0,0) rectangle (1448,1086);

\coordinate (UL) at (512,337);
\coordinate (UR) at (1046,338);
\coordinate (LO) at (512,825);

\def\RUL{136}
\def\RUR{137}
\def\RLO{138}

\draw[maincircle] (UL) circle[radius=\RUL];
\draw[maincircle] (UR) circle[radius=\RUR];
\draw[maincircle] (LO) circle[radius=\RLO];

\path[fill=white]
  (630,300)
    .. controls (690,324) and (860,324) .. (920,300)
  -- (920,369)
    .. controls (860,345) and (690,345) .. (630,369)
  -- cycle;

\draw[mainline]
  (630,300)
    .. controls (690,324) and (860,324) .. (920,300);

\draw[mainline]
  (630,369)
    .. controls (690,345) and (860,345) .. (920,369);

\path[fill=white]
  (478,462)
    .. controls (501,532) and (501,635) .. (478,706)
  -- (545,706)
    .. controls (523,635) and (523,532) .. (545,462)
  -- cycle;

\draw[mainline]
  (478,462)
    .. controls (501,532) and (501,635) .. (478,706);

\draw[mainline]
  (545,462)
    .. controls (523,532) and (523,635) .. (545,706);


\path[fill=white]
  (1033,198)
    .. controls (1040,180) and (1040,163) .. (1038,153)
  -- (1059,153)
    .. controls (1056,163) and (1056,180) .. (1059,198)
  -- cycle;

\draw[blueline]
  (1033,198)
    .. controls (1040,180) and (1040,163) .. (1038,153);

\draw[blueline]
  (1059,198)
    .. controls (1056,180) and (1056,163) .. (1059,153);

\draw[bluecircle] (1047,112) circle[radius=44];

\path[fill=white]
  (1172,384)
    .. controls (1187,392) and (1198,394) .. (1201,395)
  -- (1201,419)
    .. controls (1196,418) and (1185,414) .. (1171,407)
  -- cycle;

\draw[blueline]
  (1172,384)
    .. controls (1187,392) and (1198,394) .. (1201,395);

\draw[blueline]
  (1171,407)
    .. controls (1185,414) and (1196,418) .. (1201,419);

\draw[bluecircle] (1239,407) circle[radius=43];

\path[fill=white]
  (1950,418)
    .. controls (928,430) and (918,438) .. (909,447)
  -- (927,464)
    .. controls (935,452) and (946,444) .. (962,435)
  -- cycle;

\draw[blueline]
  (943,418)
    .. controls (928,430) and (918,438) .. (909,447);

\draw[blueline]
  (962,435)
    .. controls (946,444) and (935,452) .. (927,464);

\draw[bluecircle] (900,460) circle[radius=43];

\path[fill=white]
  (373,812)
    .. controls (360,822) and (354,828) .. (349,832)
  -- (349,856)
    .. controls (356,851) and (363,846) .. (374,838)
  -- cycle;

\draw[blueline]
  (373,812)
    .. controls (360,822) and (354,828) .. (349,832);

\draw[blueline]
  (374,838)
    .. controls (363,846) and (356,851) .. (349,856);

\draw[bluecircle] (321,842) circle[radius=42];

\path[fill=white]
  (649,814)
    .. controls (661,822) and (667,826) .. (673,831)
  -- (673,855)
    .. controls (666,850) and (661,846) .. (648,838)
  -- cycle;

\draw[blueline]
  (649,814)
    .. controls (661,822) and (667,826) .. (673,831);

\draw[blueline]
  (648,838)
    .. controls (661,846) and (666,850) .. (673,855);

\draw[bluecircle] (697,842) circle[radius=41];


\draw[redarrow]
  (888,302)
    .. controls (895,222) and (935,166) .. (990,139);

\draw[redarrow]
  (1112,140)
    .. controls (1192,185) and (1232,270) .. (1236,340);

\draw[redarrow]
  (1193,452)
    .. controls (1112,514) and (1022,525) .. (940,496);

\draw[reddouble]
  (425,845)
    .. controls (482,823) and (548,823) .. (598,845);


\node[biglabel, anchor=west] at (51,530)
  {Central $\cp^{2}$s};

\draw[blackarrow] (225,466) -- (365,340);
\draw[blackarrow] (229,592) -- (365,718);

\node[bluelabel, anchor=west] at (1014,762)
  {Peripheral $\cp^{2}$s};

\draw[bluearrow] (989,772) -- (768,834);

\draw[bluearrow] (1050,720) -- (940,520);


\node[point] at (640,335) {};
\node[point] at (920,335) {};
\node[point] at (512,462) {};
\node[point] at (512,706) {};

\node[plabel, anchor=west] at (632,282) {$P_Z$};
\node[plabel, anchor=east] at (884,282) {$P_X$};
\node[plabel, anchor=west] at (552,505) {$P_Y$};
\node[plabel, anchor=west] at (552,664) {$P_X$};

\end{tikzpicture}}

    \caption{A standard-linear action on $M_8$}
    \label{fig:cyclic} 
\end{figure}
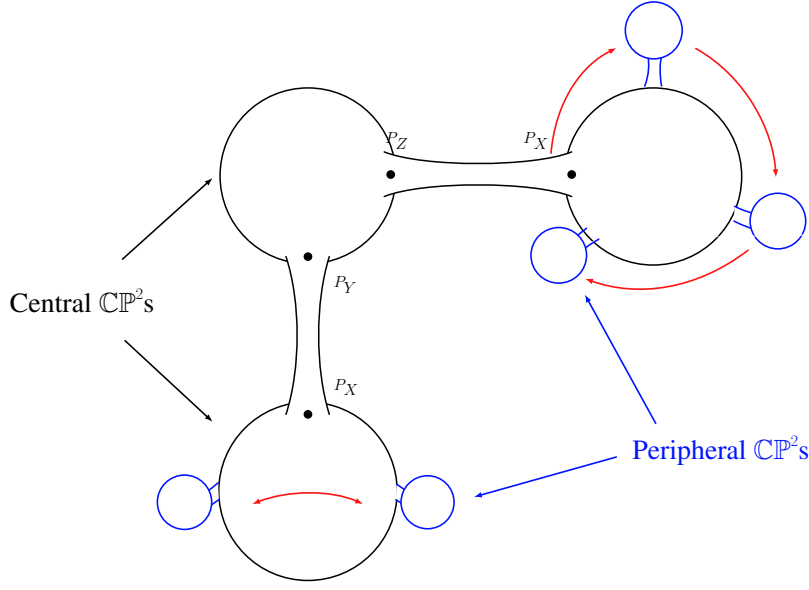

\noindent Finally, we state a (rephrased version of) a main result of Hambleton--Tanase \cite{ht}.
\begin{theorem}[cf \cite{ht, Theorem A}]\label{thm:ht-formal}
    Let $G\cong \Z/m\Z$ act faithfully by diffeomorphisms on $M_n$, for $m,n>0$ and $m$ odd.  Let $Fix(G)\subseteq M_n$ be its fixed set, and let $\rho_G:H_2(M_n;\Z)\to H_2(M_n;\Z)$ be its permutation action on homology.  Then there exists a standard-linear action of $H\cong \Z/m\Z$ on $M_n$ with $Fix(H)$ its fixed set and $\rho_H:H_2(M_n;\Z)\to H_2(M_n;\Z)$ its action on homology, such that the following two properties hold. \begin{enumerate}
        \item There is a diffeomorphism $\phi:Fix(G)\to Fix(H)$. Furthermore, $\phi$ can be extended to a map $\tilde\phi:N_G\to N_H$ between tubular neighborhoods $N_G$ of $Fix(G)\subseteq M_n$ and $N_H$ of $Fix(H)\subseteq M_n$, such that at each $p\in Fix(G)$, $\tilde \phi$ induces an isomorphism between the actions $G\to GL(T_pN_G)$ and $H\to GL(T_{\tilde \phi(p)}N_H)$.
        \item Additionally, $\rho_G=\rho_H$.
    \end{enumerate}
\end{theorem}

\noindent In other words, \Cref{thm:ht-formal} shows that if an element $\sigma\in S_n< O(H_{M_n})$ of odd order $m$ can be realized at all in $\diff^+(M_n)$, then it can be realized \textit{standard-linearly}.  Furthermore, the original realization and the standard-linear realization have the same fixed set data and homological actions.  Hambleton--Tanase prove \Cref{thm:ht-formal} by studying an equivariant Yang-Mills moduli space and constructing a cobordism between the two realizations.

\vspace{5mm}

\noindent The following lemma and corollary provide the key reduction for proving \Cref{thm:asym}, as it puts a necessary condition on any realizable permutation.  Then \Cref{prop:theta-prime-cycles} and \Cref{prop:odd-partition-prime-parts-small-fixed} quantify the proportion of partitions and permutations which satisfy this condition, respectively.  The proof of \Cref{thm:asym} can then be assembled from these results.\\

We now extract from Hambleton--Tanase \cite{ht} a simpler necessary condition for realizability in $\diff^+(M_n)$ which is strong enough for the asymptotic result below. Let $\sigma\in S_n$ have odd order $m$, and write $\pi=\langle \sigma\rangle\cong C_m.$ If $\sigma$ has a cycle of length $d$, then the corresponding transitive $\mathbb Z[\pi]$-summand
of the permutation module  (i.e., $H_2(M_n;\Z)$ with the action of $C_m$ by permuting the standard basis) is $\mathbb Z[\pi/\pi_d]$, where $\pi_d\le \pi$ is the unique subgroup of order $m/d$. Let $C_{d,n}(\sigma)$ denote the number of $d$-cycles in a permutation in $\sigma\in S_n$, and $C_{1,n}(\lambda)$ similarly for partitions $\lambda$ of $n$.  Define
\[S(\sigma):=\{\pi_d\le \pi:\ C_{d,n}(\sigma)>0\},\]
and let $\mu(\sigma)$ denote the number of maximal proper elements of $S(\sigma)$, where maximality
is with respect to inclusion among elements of $S(\sigma)\setminus\{\pi\}$.

\begin{lem}\label{lem:HT-maximal-stabilizers}
Suppose $\sigma\in S_n<O(H_{M_n})$ has odd order and is realizable in
$\mathrm{Diff}^+(M_n)$. Then
\[
\mu(\sigma)\le \max\{2,3C_{1,n}(\sigma)\}.
\]
\end{lem}

\begin{proof}
Let $\pi=\langle\sigma\rangle\cong C_m$. By \Cref{thm:ht-formal}, the given smooth
$\pi$-action has the same induced permutation action on $H_2(M_n;\mathbb Z)$ as a standard linear action. Hambleton---Tanase encode such
equivariant connected sums by admissible weighted trees $T$ which exactly characterize standard linear actions as summarized above.  Following \cite{ht}, write $H_2(M_n;\mathbb Z)$ as a
permutation module
\[
H_2(M_n;\mathbb Z)
=
\bigoplus_{\pi_\alpha\in S}\mathbb Z[\pi/\pi_\alpha]^{k_\alpha},
\]
where $S$ is the set of stabilizer subgroups for the standard basis elements of $H_2(M_n;\mathbb Z)$.
Since $X(T)$ has the same permutation action on homology as the original action, this set $S$ is
precisely $S(\sigma)$.

First suppose $\pi\notin S(\sigma)$, or equivalently $C_{1,n}(\sigma)=0$; that is, there are no central copies of $\cp^2$. Lemma 1.13 of \cite{ht} says that in this case $S$ has at most two maximal elements. Since $\pi\notin S$, all maximal elements of $S$ are proper subgroups of $\pi$. Hence $\mu(\sigma)\le 2$.

Now suppose $\pi\in S(\sigma)$, let $
h=C_{1,n}(\sigma)$; in other words, there are $h>0$ central $\cp^2$s.  Let $\tau\in S(\sigma)\setminus\{\pi\}$ be a maximal proper subgroup, and suppose it is nontrivial (f any such $\tau$ is trivial, the statement follows immediately).  By Lemma 1.12 of \cite{ht}, proper nontrivial subgroups in $S$ arise from edges\[e:v_i\to v_j\] such that $m_i>m_j=|\tau|$ and such that the fixed set $M_n^{\ang{\tau}}$ contains a $2$-sphere at the vertex $v_j$. Since
$m_j\mid m_i$, the subgroup $\tau$ is contained in the subgroup $\pi_i$ associated to $v_i$.
If $\pi_i\ne\pi$, then $\pi_i$ is a proper element of $S(\sigma)$ strictly containing $\tau$,
contradicting maximality of $\tau$. Therefore $\pi_i=\pi$, and so $v_i$ is a vertex of the central subtree $T_0$.

By the definition of the edge operation in the admissible weighted tree, the attaching point on the
initial side is chosen in the exact isotropy stratum corresponding to $\tau$. Therefore $\tau$ must
occur as an actual proper isotropy subgroup on the linear $\mathbb{CP}^2$-vertex $v_i\in T_0$.
A linear action on $\mathbb{CP}^2$ has at most three nontrivial proper isotropy groups, corresponding
to the three invariant coordinate lines. Hence each vertex of $T_0$ accounts for at most three
nontrivial maximal proper elements of $S(\sigma)$. Since $T_0$ has $h=C_{1,n}(\sigma)$ vertices,
we get $\mu(\sigma)\le 3h=3C_{1,n}(\sigma).$
\end{proof}

\noindent Define
\[P_n(\sigma)\coloneqq\#\{p\le(\log n)^2:\ p\text{ is an odd prime and }C_{p,n}(\sigma)>0\}.\] The restriction that $p\leq (\log n)^2$ is not actually needed for the following corollary, but it will be useful computationally later on to prove \Cref{prop:theta-prime-cycles}.

\begin{cor}\label{cor:prime-cycle-obstruction}
If $\sigma\in S_n<O(H_{M_n})$ has odd order and is realizable in $\mathrm{Diff}^+(M_n)$, then
\[
P_n(\sigma)\le \max\{2,3C_{1,n}(\sigma)\}.
\]
\end{cor}

\begin{proof}
If $C_{p,n}(\sigma)>0$, then $S(\sigma)$ contains the subgroup of $\pi=\langle\sigma\rangle$
of order $m/p$. This subgroup has prime index $p$ in $\pi$, hence is maximal proper in the
subgroup lattice of $\pi$. Therefore it is maximal proper in $S(\sigma)$. Distinct primes $p$ give
distinct subgroups of orders $m/p$. Thus $P_n(\sigma)\le \mu(\sigma)$, and the result now follows from Lemma \ref{lem:HT-maximal-stabilizers}.
\end{proof}

\noindent For $\theta\in\left\{1,\frac12\right\},$ let $\mathbb P_{n,\theta}$ be the probability distribution on odd order permutations $\sigma\in S_n$
given by
\[\mathbb P_{n,\theta}(\sigma)=
\frac{\theta^{c(\sigma)}}{\sum_{\tau\in S_n,\ \tau\text{ odd order}}\theta^{c(\tau)}},\] where $c(\sigma)$ denotes the total number of cycles of $\sigma$. The case $\theta=1$ is the uniform distribution on odd order elements of $S_n$, and the case $\theta=\frac{1}{2}$ will be useful for \Cref{thm:odd-ab}. The following two propositions provide the key estimates for proving \Cref{thm:asym}.

\begin{prop}\label{prop:theta-prime-cycles}
For $\theta\in\{1,\frac12\}$, under $\mathbb P_{n,\theta}$, $P_n(\sigma)\to\infty$ in probability, while $C_{1,n}(\sigma)$ is bounded above in probability in the following sense:
if $u_n\to\infty$ for $u_n\in\Z$, then $\mathbb P_{n,\theta}(C_{1,n}\ge u_n)\to0.$
Consequently,
\[
\mathbb P_{n,\theta}\left(P_n(\sigma)\le \max\{2,3C_{1,n}(\sigma)\}\right)\to0.
\]
\end{prop}

\begin{prop}\label{prop:odd-partition-prime-parts-small-fixed}
Let $t:\mathbb N\to \mathbb Z_{\ge0}$  satisfy $t(n)=o\left(\sqrt n/\log n\right).$
Let $\lambda$ be chosen uniformly among partitions of $n$ into odd parts satisfying $C_{1,n}(\lambda)\le t(n).$
Fix constants $0<a<b$, and define \[\mathcal P_n:=\{p:\ p\text{ prime and }a\sqrt n\le p\le b\sqrt n\},\ \ \ R_n(\lambda):=\#\{p\in\mathcal P_n:\ C_{p,n}(\lambda)>0\}.\]
Then
\[
\mathbb P\left(R_n(\lambda)\le \max\{2,3C_{1,n}(\lambda)\}\right)\to0.
\]
Equivalently, with probability tending to $1$, $R_n(\lambda)>\max\{2,3C_{1,n}(\lambda)\}.$
\end{prop}

\noindent We prove \Cref{prop:theta-prime-cycles} and \Cref{prop:odd-partition-prime-parts-small-fixed} in the following two subsections, but first we use them to prove \Cref{thm:asym}.

\begin{proof}[Proof of \Cref{thm:asym}, assuming \Cref{prop:theta-prime-cycles} and \Cref{prop:odd-partition-prime-parts-small-fixed}]
First we prove part 1.  Choose $\sigma\in S_n$ uniformly from the set of odd order elements. This is the distribution
$\mathbb P_{n,1}$. By Corollary \ref{cor:prime-cycle-obstruction}, every realizable odd order
$\sigma$ satisfies $P_n(\sigma)\le \max\{2,3C_{1,n}(\sigma)\}.$ But Proposition \ref{prop:theta-prime-cycles} with $\theta=1$ says that
\[\mathbb P_{n,1}\left(P_n(\sigma)\le \max\{2,3C_{1,n}(\sigma)\}\right)\to0.\]
Therefore the probability that a uniformly random odd order $\sigma\in S_n$ is realizable in
$\mathrm{Diff}^+(M_n)$ tends to $0$.

Now, we prove part 2. Let $\lambda$ be chosen uniformly among odd partitions of $n$ satisfying $C_{1,n}(\lambda)\le t(n).$
By Lemma~\ref{lem:HT-maximal-stabilizers}, every realizable odd order conjugacy class satisfies $\mu(\lambda)\le \max\{2,3C_{1,n}(\lambda)\}.$
Hence every realizable odd order conjugacy class satisfies $R_n(\lambda)\le \max\{2,3C_{1,n}(\lambda)\}$ because every prime part of size $p$ counted by $R_n(\lambda)$ contributes a distinct maximal proper stabilizer subgroup of $\pi$ of index $p$.  But Proposition~\ref{prop:odd-partition-prime-parts-small-fixed} says that, among odd partitions
with $C_{1,n}(\lambda)\le t(n)$, this inequality holds with probability tending to $0$. Therefore the
realizable proportion among such conjugacy classes tends to $0$.
\end{proof}

\subsection{Proving \Cref{prop:theta-prime-cycles}}
Let
\[
a_n^{(\theta)}
:=
\sum_{\substack{\sigma\in S_n\\ \sigma\text{ odd order}}}\theta^{c(\sigma)}.
\]
Then
\[
\sum_{n\ge0}\frac{a_n^{(\theta)}}{n!}x^n
=
\exp\left(\theta\sum_{\substack{k\ge1\\ k\text{ odd}}}\frac{x^k}{k}\right)
=
\left(\frac{1+x}{1-x}\right)^{\theta/2}.
\]
Let
\[
\alpha_n^{(\theta)}
:=
\frac{a_n^{(\theta)}}{n!}
=
[x^n]\left(\frac{1+x}{1-x}\right)^{\theta/2}.
\]
For $0\le s\le n$, define
\[
A_{n,s}^{(\theta)}
:=
(n)_s\frac{a_{n-s}^{(\theta)}}{a_n^{(\theta)}}.
\]
Since $a_n^{(\theta)}=n!\alpha_n^{(\theta)}$, this is equivalently
\[
A_{n,s}^{(\theta)}
=
\frac{\alpha_{n-s}^{(\theta)}}{\alpha_n^{(\theta)}}.
\]

\begin{lem}\label{lem:theta-coeff}
Uniformly for $\theta\in\{1,\frac12\}$, we have
\[
A_{n,s}^{(\theta)}
=
1+O(s/n)+O(n^{-\theta})
\]
for $0\le s\le n/2$. In particular, $A_{n,s}^{(\theta)}=O(1)$ for $0\le s\le n/2$. Moreover,
\[
A_{n,s}^{(\theta)}=O(n)
\]
for all $0\le s\le n$.
\end{lem}

\begin{proof}
Let $\gamma=\theta/2,$
so $\gamma\in\{\frac12,\frac14\}$. We estimate the coefficients of
\[F_\theta(x)=\left(\frac{1+x}{1-x}\right)^\gamma.\]
Near $x=1$,
\[F_\theta(x)=(1+x)^\gamma(1-x)^{-\gamma}=2^\gamma(1-x)^{-\gamma}+O\bigl((1-x)^{1-\gamma}\bigr).\]
The binomial coefficient estimate gives
\[[x^n](1-x)^{-\gamma}=\frac{\Gamma(n+\gamma)}{\Gamma(\gamma)\Gamma(n+1)}=\frac{1}{\Gamma(\gamma)}n^{\gamma-1}+O(n^{\gamma-2}).\]
The singularity at $x=-1$ contributes only $O(n^{-\gamma-1})$, since near $x=-1$ the factor
$(1-x)^{-\gamma}$ is analytic and nonzero while $(1+x)^\gamma$ has exponent $+\gamma$.
Therefore
\[\alpha_n^{(\theta)}=\frac{2^\gamma}{\Gamma(\gamma)}n^{\gamma-1}+O(n^{\gamma-2})+O(n^{-\gamma-1}).\]
Equivalently,
\[\alpha_n^{(\theta)}=\frac{2^\gamma}{\Gamma(\gamma)}n^{\gamma-1}\left(1+O(n^{-\theta})\right),\]
because $\theta=2\gamma$.  If $0\le s\le n/2$, then
\[A_{n,s}^{(\theta)}=\frac{\alpha_{n-s}^{(\theta)}}{\alpha_n^{(\theta)}}=\left(\frac{n-s}{n}\right)^{\gamma-1}\left(1+O(n^{-\theta})\right)=1+O(s/n)+O(n^{-\theta}).\]
Finally, for arbitrary $0\le s\le n$, note that $0<\theta\le1$, so $a_{n-s}^{(\theta)}\le (n-s)!$, and hence\ $\alpha_{n-s}^{(\theta)}\le1$.
Since $\alpha_n^{(\theta)}\gg n^{\gamma-1}$, we get
\[A_{n,s}^{(\theta)}=\frac{\alpha_{n-s}^{(\theta)}}{\alpha_n^{(\theta)}}\ll n^{1-\gamma}\ll n.\]
\end{proof}

\noindent Now let $k_1,\dots,k_r$ be distinct odd integers, and let $j_1,\dots,j_r\ge0$. Set $s=\sum_{i=1}^r j_i k_i.$
The usual factorial-moment count gives
\[
\mathbb E_{n,\theta}
\left[
\prod_{i=1}^r (C_{k_i,n})_{j_i}
\right]
=
\left(\prod_{i=1}^r\left(\frac{\theta}{k_i}\right)^{j_i}\right)
A_{n,s}^{(\theta)}.
\]
Indeed, choosing a specified $k_i$-cycle contributes the usual factor $1/k_i$, and now also
contributes the cycle-weight factor $\theta$.\\

\noindent First consider $C_{1,n}$. If $u_n\to\infty$, then for $n$ large we may assume $u_n\le n/2$.
Using
\[
1_{\{C_{1,n}\ge u_n\}}\le \frac{(C_{1,n})_{u_n}}{u_n!},
\]
we get
\[
\mathbb P_{n,\theta}(C_{1,n}\ge u_n)
\le
\frac{\mathbb E_{n,\theta}[(C_{1,n})_{u_n}]}{u_n!}
=
\frac{\theta^{u_n}A_{n,u_n}^{(\theta)}}{u_n!}
\to0.
\]
Thus $C_{1,n}$ is bounded above in probability. Now let
\[
\mathcal P_n=\{p\le(\log n)^2:\ p\text{ is an odd prime}\},
\]
and for $p\in\mathcal P_n$, define $I_p\coloneqq1_{\{C_{p,n}>0\}}$.
Then $P_n=\sum_{p\in\mathcal P_n}I_p.$

We first estimate $\mathbb P_{n,\theta}(C_{p,n}=0)$. Inclusion-exclusion and the factorial-moment
formula give
\[
\mathbb P_{n,\theta}(C_{p,n}=0)
=
\sum_{j=0}^{\lfloor n/p\rfloor}
(-1)^j
\frac{\theta^j A_{n,jp}^{(\theta)}}{p^j j!}.
\]
Comparing with
\[
e^{-\theta/p}
=
\sum_{j=0}^\infty
(-1)^j\frac{\theta^j}{p^j j!},
\]
and using Lemma \ref{lem:theta-coeff} in the range $jp\le n/2$, gives
\[
\mathbb P_{n,\theta}(C_{p,n}=0)
=
e^{-\theta/p}
+
O(n^{-\theta})
\]
uniformly for $p\in\mathcal P_n$. The tail $jp>n/2$ is negligible because $p\le(\log n)^2$,
while $j!$ grows faster than any power of $n$, and so the bound
$A_{n,jp}^{(\theta)}=O(n)$ suffices.\\

\noindent Similarly, for distinct $p,q\in\mathcal P_n$,
\[
\mathbb P_{n,\theta}(C_{p,n}=C_{q,n}=0)
=
e^{-\theta/p-\theta/q}
+
O(n^{-\theta}).
\]

\noindent Then
\[
\mathbb E_{n,\theta}[I_p]
=
1-e^{-\theta/p}+O(n^{-\theta})
=
\frac{\theta}{p}+O(p^{-2})+O(n^{-\theta}).
\]
Summing over $p\in\mathcal P_n$, Mertens's theorem gives
\[
\mathbb E_{n,\theta}[P_n]
=
\sum_{p\in\mathcal P_n}
\left(\frac{\theta}{p}+O(p^{-2})+O(n^{-\theta})\right)
=
\theta\log\log\log n+O(1).
\]
In particular, $\mathbb E_{n,\theta}[P_n]\to\infty.$ For the variance, since $I_p$ is an indicator,
\[
\sum_{p\in\mathcal P_n}\operatorname{Var}(I_p)
\le
\mathbb E_{n,\theta}[P_n].
\]
The two-variable estimate gives
\[
\operatorname{Cov}(I_p,I_q)=O(n^{-\theta})
\]
uniformly for distinct $p,q\in\mathcal P_n$. Thus
\[
\operatorname{Var}_{n,\theta}(P_n)
\le
\mathbb E_{n,\theta}[P_n]
+
O(|\mathcal P_n|^2 n^{-\theta})
=
O(\mathbb E_{n,\theta}[P_n]).
\]
By Chebyshev's inequality,
\[
\mathbb P_{n,\theta}
\left(
P_n\le \frac12\mathbb E_{n,\theta}[P_n]
\right)
\le
\frac{4\operatorname{Var}_{n,\theta}(P_n)}
{\mathbb E_{n,\theta}[P_n]^2}
\to0.
\]
Hence $P_n\to\infty$ in probability.\\

\noindent Finally, set
\[
u_n:=\left\lfloor \frac18\mathbb E_{n,\theta}[P_n]\right\rfloor.
\]
Then $u_n\to\infty$. With probability tending to $1$,
\[
C_{1,n}<u_n
\]
and
\[
P_n>\frac12\mathbb E_{n,\theta}[P_n].
\]
For large $n$,
\[
\frac12\mathbb E_{n,\theta}[P_n]>\max\{2,3u_n\}.
\]
Thus, with probability tending to $1$,
\[
P_n>\max\{2,3C_{1,n}\}.
\]
Equivalently,
\[
\mathbb P_{n,\theta}
\left(P_n\le \max\{2,3C_{1,n}\}\right)\to0.
\]

\subsection{Proving \Cref{prop:odd-partition-prime-parts-small-fixed}}

Let $q_{\ge3}(N)$ denote the number of partitions of $N$ into odd parts all of which are at
least $3$. Thus
\[
\sum_{N\ge0}q_{\ge3}(N)x^N
=
\prod_{\substack{j\ge3\\ j\text{ odd}}}\frac{1}{1-x^j}.
\]
We first record the coefficient estimate needed below.

\begin{lem}\label{lem:q-ge3-ratio}
Let $A=\pi/\sqrt3$. Then
\[
q_{\ge3}(N)
=
K N^{-5/4}e^{A\sqrt N}\left(1+O(N^{-1/2})\right)
\]
for some constant $K>0$. Consequently, for every fixed $B>0$, uniformly for
$0\le s\le B\sqrt N$,
\[
\frac{q_{\ge3}(N-s)}{q_{\ge3}(N)}
=
\exp\left(-\frac{A s}{2\sqrt N}\right)
\left(1+O_B(N^{-1/2})\right).
\]
\end{lem}
\begin{proof}
Let $q_{\mathrm{odd}}(N)$ denote the number of partitions of $N$ into odd parts. The classical
asymptotic for odd partitions gives
\[
q_{\mathrm{odd}}(N)
=
K_0 N^{-3/4}e^{A\sqrt N}\left(1+O(N^{-1/2})\right),
\qquad A=\pi/\sqrt3.
\]
Since excluding parts of size $1$ corresponds to multiplying the generating function by $1-x$,
we have
\[
q_{\ge3}(N)=q_{\mathrm{odd}}(N)-q_{\mathrm{odd}}(N-1).
\]
Moreover,
\[
\frac{q_{\mathrm{odd}}(N-1)}{q_{\mathrm{odd}}(N)}
=
1-\frac{A}{2\sqrt N}+O(N^{-1}).
\]
Hence
\[
q_{\ge3}(N)
=
K N^{-5/4}e^{A\sqrt N}\left(1+O(N^{-1/2})\right)
\]
for some $K>0$.

Now let $0\le s\le B\sqrt N$. Then
\[
\frac{q_{\ge3}(N-s)}{q_{\ge3}(N)}
=
\left(\frac{N-s}{N}\right)^{-5/4}
\exp\left(A(\sqrt{N-s}-\sqrt N)\right)
\left(1+O_B(N^{-1/2})\right).
\]
Since
\[
\left(\frac{N-s}{N}\right)^{-5/4}=1+O_B(N^{-1/2})
\]
and
\[
\sqrt{N-s}-\sqrt N
=
-\frac{s}{2\sqrt N}+O_B(N^{-1/2}),
\]
we get
\[
\frac{q_{\ge3}(N-s)}{q_{\ge3}(N)}
=
\exp\left(-\frac{A s}{2\sqrt N}\right)
\left(1+O_B(N^{-1/2})\right).
\]
\end{proof}

\noindent We now prove \Cref{prop:odd-partition-prime-parts-small-fixed}. Fix $0\le r\le t(n)$, and condition on the event $C_{1,n}(\lambda)=r.$
After deleting the $r$ parts of size $1$, the remaining partition is uniformly distributed among
partitions of $N=n-r$ into odd parts all of which are at least $3$. Since
\[
r\le t(n)=o\left(\frac{\sqrt n}{\log n}\right),
\]
we have $N\sim n$, uniformly for $0\le r\le t(n)$.

Let $X_{n,r}$ denote the random variable $R_n(\lambda)$ under this conditional distribution.
For $p\in\mathcal P_n$, the probability that $p$ appears as a part is
\[
\mathbb P(C_{p,n}(\lambda)>0\mid C_{1,n}(\lambda)=r)
=
\frac{q_{\ge3}(N-p)}{q_{\ge3}(N)}.
\]
By Lemma~\ref{lem:q-ge3-ratio}, uniformly for $p\in\mathcal P_n$ and $0\le r\le t(n)$,
\[
\frac{q_{\ge3}(N-p)}{q_{\ge3}(N)}
=
\exp\left(-\frac{A p}{2\sqrt N}\right)
\left(1+O(n^{-1/2})\right).
\]
Since $a\sqrt n\le p\le b\sqrt n$ and $N\sim n$, this probability is bounded above and below
by positive constants depending only on $a,b$. Therefore, by the prime number theorem,
\[
\mathbb E[X_{n,r}]
=
\sum_{p\in\mathcal P_n}
\frac{q_{\ge3}(N-p)}{q_{\ge3}(N)}
\asymp
\#\mathcal P_n
\asymp
\frac{\sqrt n}{\log n},
\]
uniformly for $0\le r\le t(n)$.

Next we estimate the variance. For distinct $p,q\in\mathcal P_n$,
\[
\mathbb P(C_{p,n}(\lambda)>0,\ C_{q,n}(\lambda)>0\mid C_{1,n}(\lambda)=r)
=
\frac{q_{\ge3}(N-p-q)}{q_{\ge3}(N)}.
\]
Again by Lemma~\ref{lem:q-ge3-ratio},
\[
\frac{q_{\ge3}(N-p-q)}{q_{\ge3}(N)}
=
\exp\left(-\frac{A(p+q)}{2\sqrt N}\right)
\left(1+O(n^{-1/2})\right),
\]
while
\[
\frac{q_{\ge3}(N-p)}{q_{\ge3}(N)}
\frac{q_{\ge3}(N-q)}{q_{\ge3}(N)}
=
\exp\left(-\frac{A(p+q)}{2\sqrt N}\right)
\left(1+O(n^{-1/2})\right).
\]
Thus, uniformly for distinct $p,q\in\mathcal P_n$,
\[
\operatorname{Cov}
\left(
1_{\{C_{p,n}>0\}},
1_{\{C_{q,n}>0\}}
\mid C_{1,n}=r
\right)
=
O(n^{-1/2}).
\]
Since the variables $1_{\{C_{p,n}>0\}}$ are indicators,
\[
\sum_{p\in\mathcal P_n}
\operatorname{Var}(1_{\{C_{p,n}>0\}}\mid C_{1,n}=r)
\le
\mathbb E[X_{n,r}].
\]
Therefore
\[
\operatorname{Var}(X_{n,r})
\le
\mathbb E[X_{n,r}]
+
O\left((\#\mathcal P_n)^2n^{-1/2}\right).
\]
Using
\[
\#\mathcal P_n\asymp\frac{\sqrt n}{\log n},
\]
we obtain
\[
\operatorname{Var}(X_{n,r})
=
O\left(\frac{\sqrt n}{\log n}\right)
+
O\left(\frac{\sqrt n}{(\log n)^2}\right)
=
O\left(\frac{\sqrt n}{\log n}\right),
\]
uniformly for $0\le r\le t(n)$. Since
\[
\mathbb E[X_{n,r}]\asymp\frac{\sqrt n}{\log n},
\]
Chebyshev's inequality gives
\[
\mathbb P\left(
X_{n,r}\le \frac12\mathbb E[X_{n,r}]
\mid C_{1,n}=r
\right)\to0
\]
uniformly for $0\le r\le t(n)$. Hence, uniformly for $0\le r\le t(n)$,
\[
X_{n,r}\gg \frac{\sqrt n}{\log n}
\]
with probability tending to $1$.

Since
\[
t(n)=o\left(\frac{\sqrt n}{\log n}\right),
\]
we get, uniformly for $0\le r\le t(n)$,
\[
\mathbb P\left(
X_{n,r}\le \max\{2,3t(n)\}
\mid C_{1,n}=r
\right)\to0.
\]
Averaging over $0\le r\le t(n)$, this gives
\[
\mathbb P\left(
R_n(\lambda)\le \max\{2,3t(n)\}
\ \middle|\ 
C_{1,n}(\lambda)\le t(n)
\right)\to0.
\]
Since $C_{1,n}(\lambda)\le t(n)$ under the conditioning, we have
\[
\max\{2,3C_{1,n}(\lambda)\}\le \max\{2,3t(n)\}.
\]
Therefore
\[
\mathbb P\left(
R_n(\lambda)\le \max\{2,3C_{1,n}(\lambda)\}
\ \middle|\ 
C_{1,n}(\lambda)\le t(n)
\right)\to0.
\]

\subsection{Proving \Cref{thm:2groups}: 2-subgroups of $S_n<O(H_{M_n})$}\label{sec:2group}
In this subsection, let $G\leq S_n< O(H_{M_n})$ with $\abs{G}=2^m$.

\begin{prop}\label{prop:2-groups}
    If $G\leq S_n< O(H_{M_n})$ with $\abs{G}=2^m$ is realizable in $\diff^+(M_n)$, then there is a subgroup $H\leq G$ such that $\log_2\abs{H}\geq m-\log_2(n)$ and either \begin{enumerate}
        \item $H$ embeds into $SO(4)$
        \item there is an exact sequence \[1\to K\to H\to L\to 1\] with $K$ cyclic, and $L$ embeds into $O(3)$.
    \end{enumerate}If $n$ is odd, then $G$ embeds into $SO(4)$.
\end{prop}
\begin{proof}
    Assume $G$ is realizable in $\diff^+(M_n)$.  First, no order 2 element $\phi\in G$ has a realizing representative $f\in\diff^+(M_n)$ which acts freely on $M_n$, for any $n\geq 1$. For contradiction, assume $\phi\in G$ has a realization $f\in\diff^+(M_n)$ acting freely on $M_n$ with $\abs{\phi}=2$.  Let $k$ denote the number of 2-cycles in $\phi$ as a permutation.  In particular, $\dim H_2(M_n;\R)^{\ang{\phi}}=n-k$.  Then the $G$-signature theorem, \Cref{thm:gsig}, gives \[n-2k=2(n-k)-n=\sum_z\de_z+\sum_C\de_C=0.\]  Thus $n$ must be even, and $\phi$ must have a cycle type of only 2-cycles.  But now in the language of \Cref{prop:ed}, $\chi(M_n^{\ang{f}})=t-c+2=0-0+2\neq0$, and so $M_n^{\ang{f}}$ cannot be empty.

    Since $G$ is a $2$-group, it has nontrivial center. Let $\phi\in Z(G)$ have order $2$.  Since no realization $f\in\diff^+(M_n)$ of $\phi$ acts freely, by \Cref{prop:ed}, $M^{\ang{f}}$ consists of $k$ isolated fixed points and $s$ 2-spheres.  If $k>0$, then by \Cref{lem:rankorbitsize}, there is a subgroup $H\leq G$ with \[\log_2\abs{H}\geq m-v_2(k)\geq  m-\log_2(k)\] such that $H$ fixes a point $p$ and hence embeds into $SO(4)$.  Similarly, if $s>0$, then there is a subgroup $H\leq G$ with \[\log_2\abs{H}\geq m-v_2(s)\geq  m-\log_2(s)\] such that $H$ fixes some 2-sphere, setwise.  The kernel $K$ of the action of $H$ on this sphere $S$ can have at most one element of order 2:  if two involutions each act fixing a sphere $S$ pointwise, they must each act on the normal bundle $NS\subseteq TM_n|_S$ by $(-1)$, and hence they must be equal.  Since $K$ is a 2-group, it must then be cyclic since it embeds into $SO(2)$.  Observe that $L$ embeds into $O(3)$ by uniformization.  To conclude the first claim, note that $k,s\leq n$.

    If $n$ is odd, then $n-2k$ is also odd, and so $f$ fixes an isolated point.  The result follows.
\end{proof}

\begin{lem}\label{lem:2groups}
    If $G\leq S_n< O(H_{M_n})$ with $\abs{G}=2^m$ is realizable in $\diff^+(M_n)$, then $G$ has a generating set with at most $\lceil \log_2 n\rceil+5$ elements.
\end{lem}

\begin{proof}
    By \Cref{prop:2-groups}, there is a subgroup $H\leq G$ such that $\log_2\abs{H}\geq m-\log_2(n)-1$ and either $H$ embeds either into $SO(4)$, or $H/K$ for $K$ cyclic embeds into $O(3)$.  In particular, $[G:H]\leq 2n$.  By the classical classifications of finite subgroups of $O(3)$ and $SO(4)$ (for example, via its universal cover $SU(2)$), $H$ is then generated in either case by at most 4 elements.  Since $H$ and $G$ are 2-groups, inclusion $H\leq G$ can be refined to a chain of inclusions \[H\leq H_1\leq\dots\leq H_r=G,\] where $[H_{i+1}:H_i]=2$, and where \[r=\log_2[G:H]\leq \log_2(2n)=1+\log_2 n.\]  Each successive group $H_i$ can be generated by adding one additional element to a generating set for $H_{i-1}$.  Then $G$ is generated by at most $\lceil \log_2 n\rceil+5$ elements.
\end{proof}

\begin{proof}[Proof of \Cref{thm:2groups}]
    We prove the first part, about 2-groups of $S_n< O(H_{M_n})$.  Let $A_n$ denote the total number of 2-subgroups of $S_n$. Let $k_n=\lfloor n/2\rfloor$, and let $E_n$ denote the number of subgroups of some maximal elementary abelian 2-group $C_2^{k_n}\leq S_n$.  Such subgroups correspond to $\F_2$-vector spaces of $\F_2^{k_n}$.  If $k_n=2\ell$ for $\ell\in\Z$, let $a=b=\ell$.  If $k_n=2\ell+1$, let $a=\ell$ and $b=\ell+1$.  Then the graphs of distinct linear maps $\F_2^a\to \F_2^b$ give distinct subspaces of $\F_2^{k_n}$.  There are $2^{ab}$ such maps and $ab=\floor{k_n^2/4}$, and so \[\log_2 A_n\geq \log_2 E_n\geq \floor{\frac{k_n^2}{4}}=\frac{n^2}{16}+O(n).\]
        
    Let $B_n$ denote the total number of 2-groups which are generated by a set of size at most $\log_2(n)+5$. \Cref{lem:2groups} implies that any 2-group in $S_n< O(H_{M_n})$ which is realizable in $\diff^+(M_n)$ fits this criterion. Then $B_n$ is bounded from above by the number of ordered $(\lceil \log_2(n)\rceil+5)$-tuples of elements in $S_n$, which is $(n!)^{\lceil \log_2(n)\rceil+5}$.  Then $\log_2 B_n\leq (\lceil \log_2(n)\rceil+5)\log_2(n!)=O(n(\log n)^2)$ by Stirling's approximation.  Comparing bounds on $A_n$ and $B_n$ gives $B_n/A_n=O(e^{-cn^2})$ for some constant $c>0$.\\

    Now we prove the second part, about conjugacy classes.  Let $C_n$ denote the total number of conjugacy classes of 2-subgroups of $S_n$, and let $D_n$ denote the total number of conjugacy classes of 2-groups which are generated by at most $\log_2(n)+5$ elements. It suffices to show that $D_n/C_n\to 0$.  Simply observe that $\log D_n\leq \log B_n= O(n(\log n)^2)$, and $\log C_n\geq \log(A_n/n!)\geq c'n^2-O(n\log n)$ for some constant $c'$.  The result follows.

\end{proof}

\begin{rmk}
    There are precisely two non-conjugate subgroups of $O(H_{M_n})$ abstractly isomorphic to the symmetric group $S_n$.  Along with the standard copy made up of the usual $n$-by-$n$ permutation matrices, there is also the group of signed permutation matrices $H=\{(-1)^{sgn(\sigma)}\sigma\}_{\sigma\in S_n\leq O(H_{M_n})}$ for $\sigma\in S_n\leq  O(H_{M_n})$ a totally positive permutation matrix and $sgn(\sigma)$ its sign as a permutation.  Since the map $\pi:\diff^+(M_n)\to O(H_{M_n})$ is surjective for all $n\geq 1$, then Nielsen realization is a conjugacy invariant in $O(H_{M_n})$.  Furthermore, the odd order elements of $H$ coincide precisely with the odd order elements in the standard copy of $S_n\leq O(H_{M_n})$.  In particular, \Cref{thm:asym} and \Cref{thm:2groups} hold equally true considering $H$ instead of $S_n\leq O(H_{M_n})$, and hence they hold for any subgroup of $O(H_{M_n})$ abstractly isomorphic to $S_n$.
\end{rmk}

\subsection{Other realizations and obstructions in $S_n<O(H_{M_n})$}\label{sec:largeSn}

First, considering regular polytopes proves the following proposition.
\begin{prop}\label{prop:polygon}
    In the following, any embedding $S_k\times S_\ell\leq S_n<O(H_{M_n})$ is taken to be the permutations on disjoint sets of $k$ and $\ell$ out of $n$ letters.  The inclusion $D_k\leq S_k$ is also taken to be the standard embedding. \begin{enumerate} 
        \item The product of dihedral groups $D_k\times D_\ell\leq S_k\times S_\ell\leq S_{k+\ell}$ of order $4k\ell$ is realizable in $\diff^+(M_{k+\ell})$.  The product of cyclic groups $C_k\times C_\ell\leq S_k\times S_\ell\leq S_{k+\ell}\leq S_n<O(H_{M_n})$ is realizable in $\diff^+(M_n)$  for all $n\geq k+\ell$.
        \item Let $G$ be the isometry group of a platonic solid or the orientation-preserving isometries of any regular 4-dimensional polytope, embedded as a subgroup of the permutations of its $k$ vertices $G\leq S_k\leq S_n<O(H_{M_n})$.  Then $G$ is realizable in $\diff^+(M_n)$.
        
        In the case of a platonic solid, let $H$ be its orientation-preserving isometries.  Then for any $\ell>0$, the group $H\times C_\ell\leq S_k\times S_\ell\leq S_{k+\ell}$ is realizable in $\diff^+(M_{k+\ell})$.  
        
        The group $K$ of all isometries of a regular 4-dimensional polytope with $k$ vertices is realizable in $M_k$, as is the group of all orientation-preserving isometries of a 5-dimensional polytope.
        \item  Let $G$ be the orientation-preserving symmetry group of a regular 5-cube or 5-orthoplex, embedded into the permutations $S_n$ of its $n$ vertices, and hence into $O(H_{M_n})$ under the natural embedding.  Then the group $G$ is realizable in $\diff^+(M_n)$. In particular, $n=32$ or $n=10$ in this setting.
        \item The group $S_n<O(H_{M_n})$ is realizable in $\diff^+(M_n)$ for all $n\leq 5$.  The group $A_6$ is realizable in $\diff^+(M_6)$.
    \end{enumerate}
\end{prop}
\begin{proof}\hfill
    \begin{enumerate}
        \item  Model $M_{k+\ell}$ by connect summing $k$ copies of $\cp^2$ onto a central $S^4$ at regular intervals around a geodesic $S^1\subset S^4$, and $\ell$ copies at regular intervals around an orthogonal geodesic $S^1\subset S^4$.  Then rotations of $S^4$ induce an action of $D_k\times D_\ell$, which extends over connect-sums to all of $M_{k+\ell}$.

        \item Realization of $G$ in $\diff^+(M_n)$ follows simply by rigid rotations of $S^4$, by connect summing copies of $\cp^2$ in neighborhoods of the vertices of such regular polytopes on $S^4$.  Realizing $H\times C_\ell$ follows similarly, positioning the points in a platonic solid and a circle on orthogonal equatorial subspheres $S^3$ and $S^1$ in $S^4$.  Realizing $K$ again is simply achieved by rigid rotation of $S^4$.

        \item  Again, realize $G$ in $\diff^+(M_n)$ by rigid motions, using a model where the $n$ copies of $\cp^2$ are glued symmetrically at the vertices of an inscribed cube or orthoplex in $S^4$.
        
        \item For $n=2,3,4,5,6$, model $M_n$ by connect summing $n$ copies of $\cp^2$ onto a central $S^4$ at neighborhoods of the vertices of an inscribed regular $(n-1)$-simplex.  Hence by isometries of the round $S^4$, realize $A_n$, the orientation-preserving isometry group of the $(n-1)$-simplex.  For $n=2,3,4,5$, the vertex points can be chosen to lie on an equatorial $S^3\subseteq S^4$.  Hence, first reflect $S^4\subset \R^5$ through a 4-plane, which induces an orientation-reversing isometry of $S^3$, and then reflect $S^4$ through the 4-plane containing that $S^3$.  In this way realize any isometry of the $(n-1)$-simplex by an orientation-preserving diffeomorphism of $S^4$ which extends to a diffeomorphism on $M_n$ inducing the corresponding permutation.
        
    \end{enumerate}
\end{proof}

\begin{rmk}
    \Cref{prop:polygon} generalizes in several elementary but cumbersome-to-state ways.  For example, consider connect summing a copy of $M_2$ at each vertex of a platonic solid with $k$ vertices.  This then gives a realizable subgroup $S_k\leq S_{2k}\leq S_n<O(H_{M_n})$.
\end{rmk}

In the rest of this section, we prove \Cref{cor:Sn-summary}.

\begin{lem}\label{lem:cycle-restrinctions-small-cases}
    Suppose $\sigma\in S_n\leq O(H_{M_n})$ is of odd order $m$ and realizable in $\diff^+(M_n)$. \begin{enumerate}
        \item Suppose $\sigma$ has no fixed points as a permutation.  If the cycle type of $\sigma$ contains two prime lengths $p,q>1$ which are minimal under divisibility among all cycle lengths, then every cycle length in $\sigma$ other than $p$ and $q$ is divisible by $pq$.
        \item Suppose $\sigma$ has a single fixed point as a permutation. Then it cannot simultaneously have cycles of lengths $3$, $5$, and $7$.
    \end{enumerate}
\end{lem}

\begin{proof}
    By \Cref{thm:ht-formal}, it suffices to show for standard linear actions.  First suppose that $\sigma$ has no fixed points. Then the standard-linear model has empty central core, so it is built from an orthogonal $\mathbb Z/m\mathbb Z$-action on $S^4$, followed by equivariant connected sums along orbits. Write the two nontrivial rotation weights on $S^4$ as $a,b \in \mathbb Z/m\mathbb Z$, and let $o(a)=m/(a,m), o(b)=m/(b,m)$ denote their additive orders. The initial nontrivial orbit lengths on $S^4$ are $o(a), o(b),$ and $\operatorname{lcm}(o(a),o(b)).$ Since the action is faithful, $m=\operatorname{lcm}(o(a),o(b))$. Every later vertex of the admissible tree lies over an earlier one, and its orbit length is a multiple of the orbit length of its predecessor. Thus a minimal nontrivial cycle length must already occur among the initial orbit lengths above. If two distinct primes $p,q$ occur as minimal nontrivial cycle lengths, then, after relabeling the two rotation planes, $o(a)=p,o(b)=q.$ Hence $m=pq$. A branch lying over an orbit of length $p$ has stabilizer of order $q$, so all orbit lengths in that branch are either $p$ or $pq$. Similarly, a branch lying over an orbit of length $q$ has all orbit lengths either $q$ or $pq$, and a branch lying over a generic orbit has orbit length divisible by $pq$.
    
    Now assume that the central core consists of exactly one copy of $\cp^2$. A standard linear action on this copy has the form $t \mapsto \operatorname{diag}(1,\zeta^a,\zeta^b)$ in projective coordinates, and the three distinguished orbit lengths from which peripheral branches may begin are the additive orders in $\mathbb Z/m$ of $(a, b, a-b).$ If cycle lengths $3,5,7$ all occurred, then, after relabeling $a,b,a-b$, there would be elements $x,y \in \mathbb Z/m$ of additive orders $3$ and $5$ such that $x-y$ has additive order $7$. But in a cyclic group, the difference of elements of prime orders $3$ and $5$ has order $15$. Indeed the subgroup they generate is cyclic of order $15$, and $x-y$ projects nontrivially to both the order $3$ and order $5$ factors. This is a contradiction.
\end{proof}

\noindent \Cref{cor:Sn-summary} then follows by combining the following two propositions.

\begin{cor}\label{cor:Sn-to-8}
    All cyclic subgroups of $S_n<O(H_{M_n})$ are realizable in $\diff^+(M_n)$ for $1\leq n\leq 8$.  When $n\geq 15$, there exists a non-realizable cyclic subgroup.
\end{cor}

\begin{proof}
    All $\sigma\in S_n$ for $n\leq 8$ are of cycle types with at most two distinct cycle lengths greater than 1 and hence can be constructed as in case (1) of \Cref{prop:polygon}.
    
    Now say $n\geq 15$ is odd.  If $15\nmid (n-8)$, then the permutation of cycle type $(3,5,n-8)$ is not realizable in $\diff^+(M_n)$ by the first part of \Cref{lem:cycle-restrinctions-small-cases}.  If $15\mid (n-8)$, then $n\equiv 2\pmod 3$, and so $n-10\equiv 1\mod 3$.  Then $21\nmid (n-10)$.  Consider the permutation $(3,7,n-10)$ and apply \Cref{lem:cycle-restrinctions-small-cases}.

    For $n=16$, consider the permutation with cycle lengths $(1,3,5,7)$ and apply \Cref{lem:cycle-restrinctions-small-cases}.
    
    When $n\geq 18$ is even, consider the permutation with cycle lengths $(3,5,7,n-15)$ and apply \Cref{lem:cycle-restrinctions-small-cases}.
\end{proof}

\noindent For large $n$, \Cref{thm:asym} shows that realizing all of $S_n<O(H_{M_n})$ at once is impossible.  In fact, such an obstruction begins at least at $n=8$.

\begin{prop}\label{prop:sn-obstr}
    Let $n\geq8$, and pick some $k$ with $k-2 > \max\{n/2,5\}$. Consider a subgroup $S_k\leq S_n< O(H_{M_n})$ given by permuting some choice of $k$ standard basis elements of $H_2(M_n;\Z)$ (i.e., hyperplane classes of $\cp^2$s).  Then the group $S_k$ is not realizable in $\diff^+(M_n)$.  In particular, the group $S_n<O(H_{M_n})$ is not realizable in $\diff^+(M_n)$ for any $n\geq 8$.
\end{prop}

\begin{proof}
    Consider the permutation $(12)\in S_k$, and suppose it is realized by an order 2 diffeomorphism $f\in\diff^+(M_n)$. By \Cref{prop:ed} with $p=2$, we have $\beta_1(M_n^{\ang{f}})=0$, and hence $M_n^{\ang{f}}$ consists only of isolated points and spheres.  Let $\ell$ denote the number of spheres.  Since $\beta_0(S^2)+\beta_2(S^2)=2$, then $0\leq\ell\leq n/2$.  Now consider the subgroup $S_{k-2}\leq S_k$ which permutes the last $k-2$ basis vectors.

    By the $G$-signature theorem, \Cref{thm:gsig}, applied to the group $\ang{f}$, \begin{align*}
        n-2=2(n-1)-n&=\sum_{z}\de_z+\sum_{C}\de_C=\sum_{z}0+\sum_{C}C\cdot C.
    \end{align*}
    Since $n-2>0$ and the signature defect of a point is $0$ when $p=2$, then there must be 2-dimensional components.  That is, $\ell>0$.
    
    Notice then that $S_{k-2}$ acts on the finite set of $\ell$ spheres.  The kernel of this action is thus either trivial, $A_{k-2}$, or $S_{k-2}$.  Since $\ell<k-2$, the action cannot have trivial kernel.  Thus $A_{k-2}$ fixes an $f$-fixed sphere set-wise.  But if $k-2\geq 6$, then $A_{k-2}$ cannot act effectively on $S^2$, and so $A_{k-2}$ fixes that sphere pointwise.  Let $p\in S^2$ be a point on such a sphere.  For $k-2\geq 6$, the group $A_{k-2}$ does not admit an embedding into $SO(4)$.  Hence some nontrivial element $g\in A_{k-2}$ acts by the identity on $T_pM$.  Thus the group $A_{k-2}$, and hence $S_k$ cannot be realizable.
\end{proof}

\begin{rmk}
    Similar arguments to the proof of \Cref{prop:sn-obstr} can be used to show that $A_n\leq S_n<O(H_{M_n})$ is not realizable in $\diff^+(M_n)$ for any $n\geq 10$. In fact, the ideas of the proof obstruct a much wider class of subgroups from being realized.  For example, consider any $H\leq S_n<O(H_{M_n})$ which contains a 2-cycle whose centralizer in $H$ is large and has no small, nontrivial normal subgroups which embed into $SO(4)$ or $O(3)$.  We omit the most general possible formulation, which would be overly cumbersome.
\end{rmk}

\section{General subgroups}\label{sec:other}
This final section considers all subgroups of $O(H_{M_n})$, not necessarily contained in either $G_n\leq O(H_{M_n})$ or $S_n<O(H_{M_n})$.  In \Cref{sec:gen-constr}, we give constructions which in particular show cyclic realization in the smallest cases.

\begin{prop}\label{prop:m23}
    When $n=2$ or $n=3$, all cyclic subgroups of $O(H_{M_n})$ are realizable in $\diff^+(M_n)$.
\end{prop}

\noindent Finally, in \Cref{sec:odd-ab} we show the following result, which implies the remainder of parts 1 and 2 of \Cref{thm:mega-thm}.

\begin{theorem}\label{thm:odd-ab}{\bf (General asymptotic non-realizability)} \begin{enumerate}
    \item Random subgroups of $O(H_{M_n})$ are asymptotically almost never realizable in $\diff^+(M_n)$.  The same is true up to conjugacy.
    \item  Random elements of odd order in $O(H_{M_n})$ are asymptotically almost never realizable in $\diff^+(M_n)$.
    \item Along the odd subsequence $(M_{2k+1})$, random abelian subgroups of $O(H_{M_{2k+1}})$ are asymptotically almost never realizable in $\diff^+(M_{2k+1})$.  More precisely, suppose $H\leq O(H_{M_{2k+1}})$ contains some order two element $\phi$ with an odd number of $+1$s on its diagonal as a $(2k+1)$-by-$(2k+1)$ matrix, and suppose that the 2-primary part $H_2\leq H$ does not embed in $SO(4)$.  Then $H$ is not realizable in $\diff^+(M_{2k+1})$, and almost all abelian subgroups of $O(H_{M_{2k+1}})$ are of this form as $k\to\infty$.
\end{enumerate}
\end{theorem}

\subsection{Cyclic realization: small cases}\label{sec:gen-constr}

\begin{constr}\label{constr:4}
    Consider $M_2=\cp_1^2\#\cp_2^2$.  We seek to realize the order 4 mapping class \[\begin{pmatrix}
        0&-1\\1&0
    \end{pmatrix}\in O(2,\Z)\cong O(H_{M_2})\] by a diffeomorphism $f$ of order 4.  Consider coordinates $[X_1:Y_1:Z_1]$ on $\cp^2_1$ and $[X_2:Y_2:Z_2]$ on $\cp^2_2$.  Near points $[0:0:1]$ on each copy, consider local coordinates $\re X_1, \im X_1, \re Y_1, \im Y_1$ on $\cp^2_1$ and $\re X_2, \im X_2, \re Y_2, \im Y_2$ on $\cp^2_2$; call these the standard local bases for $\cp^2_1$ and $\cp_2^2$.  Form $\cp^2_1\#\cp^2_2$ by gluing locally according to the following linear map on local charts from $\cp^2_1$ to $\cp_2^2$, written as a matrix in the standard local bases: \[\begin{pmatrix}
        1&0&0&0\\0&0&1&0\\0&1&0&0\\0&0&0&1
    \end{pmatrix},\] i.e., by identifying unit balls around $[0:0:1]$ in each copy of $\cp^2$ according to this linear map.  Notice there is locally an orientation-reversing diffeomorphism $r$ of $M_2$ in a neighborhood of the connect sum by a reflection on the $S^3$ gluing locus of the connect-sum.  Specifically, it sends the point $(a,b,c,d)$ in the standard basis of $\cp^2_1$ to $(a,c,b,d)$ in the standard basis of $\cp^2_2$, and vice versa.  This map does not, however, extend globally to a diffeomorphism of $M_2$, which does not admit any orientation-reversing diffeomorphisms (since it has non-zero signature, for example).  However, notice that the map exchanging basis elements $\im X_1$ and $\re Y_1$ extends across the connect sum to exchange $\im X_2$ and $\re Y_2$.  Then compose $r$ with this reflection to obtain a map which is orientation-preserving and extends globally to a diffeomorphism $r_2:M_2\to M_2$.  In local coordinates near the connect sum, $r_2$ acts by the simple formula $(a,b,c,d)\mapsto (a,b,c,d)$ in the standard local bases of $\cp^2_1$ and $\cp_2^2$, respectively, and vice versa.

    Finally, construct the sought $f$ by composing $r_2$ with a rotation.  The map $[X_1:Y_1:Z_1]\mapsto [-X_1:Y_1:Z_1]$ on $\cp^2_1$ extends across the connect sum to $[X_2:Y_2:Z_2]\mapsto[\tilde X_2:\tilde Y_2:\tilde Z_2]$ where $\tilde X$ refers to reflection across the imaginary axis (and $\tilde Y,\tilde Z$ similarly).  Then composing this map with $r_2$ gives the desired $f$.
\end{constr}

\vspace{5mm}

\begin{constr}\label{constr:5}
    We generalize \Cref{constr:4} analogously to the generalization from \Cref{constr:2} to \Cref{constr:3}.  First note that \Cref{constr:4} can be employed around any of the three coordinate points $[1:0:0],[0:1:0],[0:0:1]$.  Start with an $M_n$ constructed as in \Cref{constr:3}.  If $n$ is even, there is a unique edge $e$ in the adjacency tree (which, recall, records which copies of $\cp^2$ are connected locally) which divides the tree into two trees of equal size.  If the tree is symmetric about this edge (meaning it admits an isometry which reflects the edge), then apply \Cref{constr:4} locally to the connect sum at $e$.  In that case, the action extends to a well-defined diffeomorphism of $M_n$.
\end{constr}

\vspace{5mm}

\begin{cor}\label{cor:m2}
    All cyclic groups in $O(H_{M_2})\cong O(2,\Z)$ are realizable in $\diff^+(M_2)$.
\end{cor}
\begin{proof}
    By \Cref{thm:large_rank} and \Cref{constr:4}, it suffices to construct lifts of the matrices \[\epsilon\begin{pmatrix}
        0&1\\1&0
    \end{pmatrix}\] for $\epsilon=\pm1$.  For $\epsilon=1$, attach the two copies of $\cp^2$ with the standard connect sum of \Cref{constr:1}, and then perform two reflections locally in a neighborhood $S^3\times [0,1]$ and then extend: one which swaps the two copies by reflecting over the gluing locus $S^3$, and the other which reflects in some standard basis element of each local coordinate chart.  Equivalently, swap the two local charts around the gluing loci by a rotation inside some ambient $\R^5$.  See \Cref{fig:m2}.  For $\epsilon=-1$, simply compose by complex conjugation on both copies.
\end{proof}

\begin{figure}[hbt!] 
    \centering
    \includegraphics[width=0.5\textwidth]{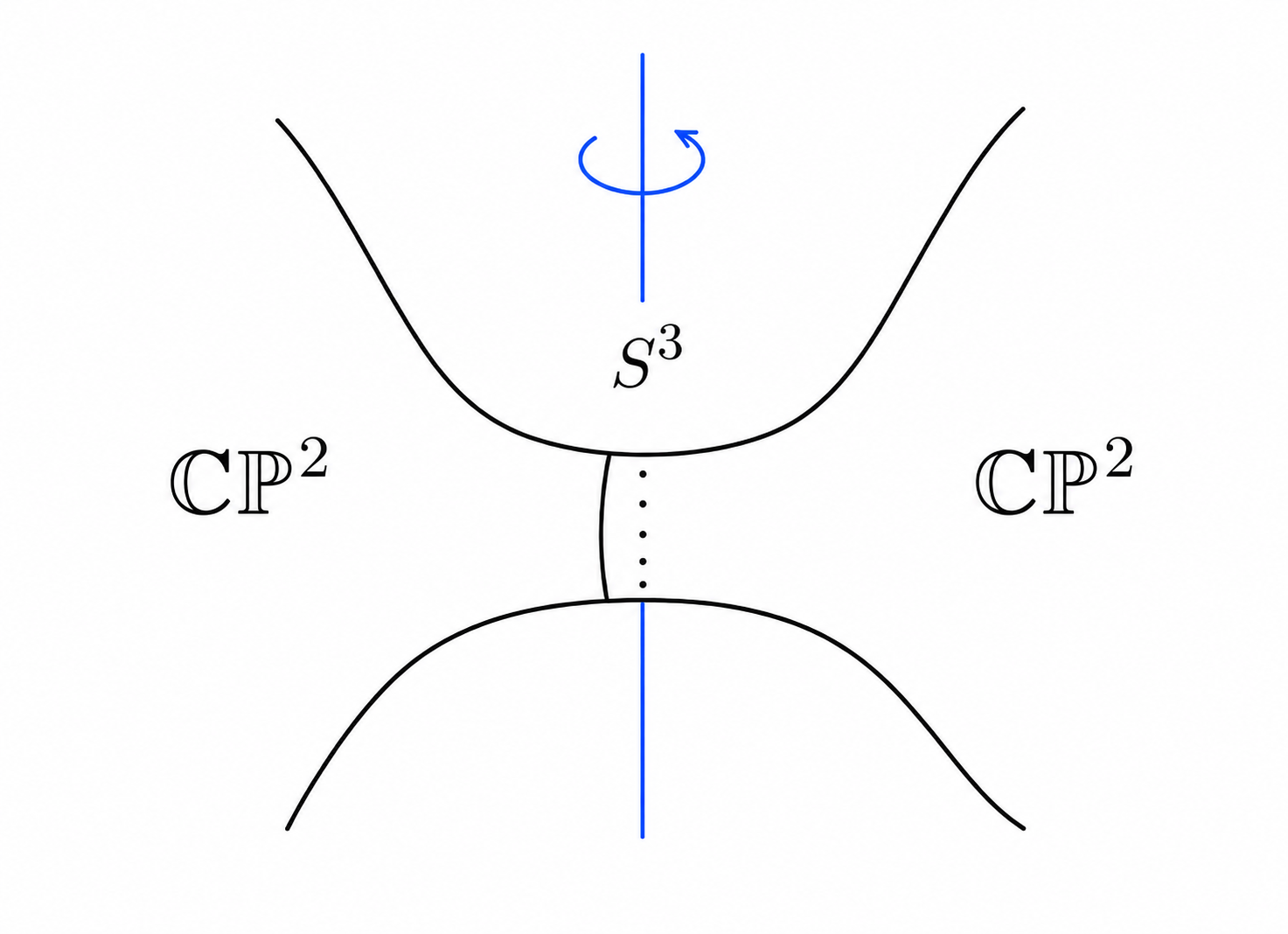}
    \caption{\Cref{cor:m2}, when $\epsilon=+1$: in a local model swap two copies of $\cp^2$ by rotation on an axis}
    \label{fig:m2} 
\end{figure}

\begin{constr}\label{constr:m3}
    We show how to realize all elements of $O(H_{M_3})$ of the form \begin{equation}
        \begin{pmatrix}
        0&\pm 1 &0\\\pm1&0&0\\
        0&0&\pm1
    \end{pmatrix}
    \end{equation} for all combinations of signs.

    The key to a simple geometric construction is to adjust our local model of connect sum.  In previous constructions, we modeled connect sum by directly gluing a copy of $S^3$ in each component manifold.  We now change this picture by adding a tube $S^3\times I$ between the two manifolds.
    
    Applying this to $\cp^2\#\cp^2$, the tube inherits a set of coordinates from the standard coordinate charts on the gluing loci of each copy of $\cp^2$. All earlier constructions on $M_2$ can be rephrased naturally in terms of this tube perspective.  In realizing any matrix of the form \[\begin{pmatrix}
      0&\pm 1\\\pm1&0  
    \end{pmatrix}\] on $\cp^2\#\cp^2$ as in \Cref{constr:4} and the proof of \Cref{cor:m2}, there is a fixed point $p\in S^3\times\{\frac{1}{2}\}\subset S^3\times I$.  Attach a third copy of $\cp^2$ in an invariant neighborhood around such $p$, gluing along standard projective coordinates.  A choice of parity in this identification then completes the construction.  See \Cref{fig:m3}.
\end{constr}

\begin{figure}[hbt!] 
    \centering
    \includegraphics[width=0.7\textwidth]{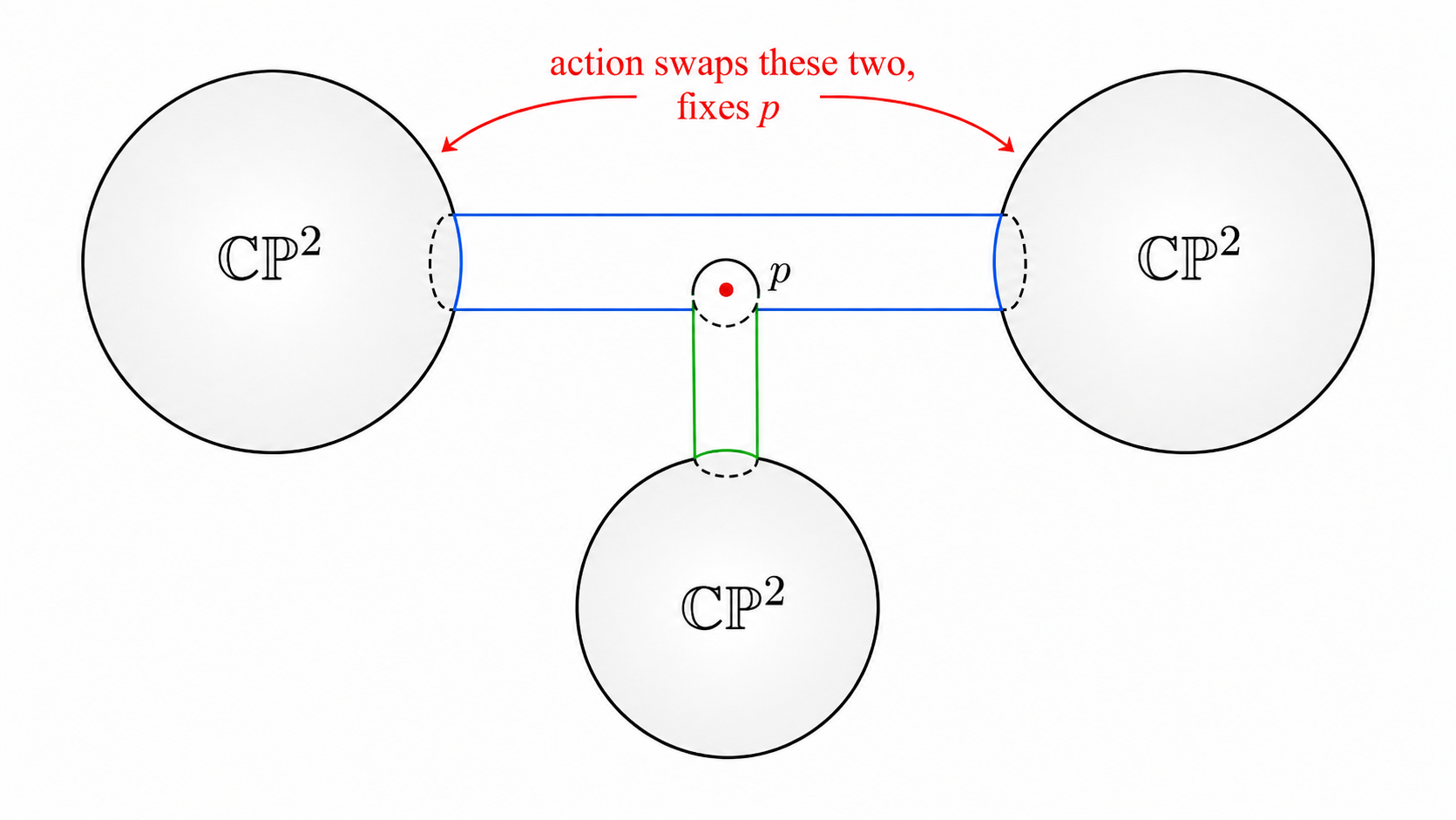}
    \caption{\Cref{constr:m3}}
    \label{fig:m3} 
\end{figure}

\vspace{5mm}

\begin{cor}\label{cor:m3}
    All cyclic groups in $O(H_{M_3})\cong O(3,\Z)$ are realizable in $\diff^+(M_3)$.
\end{cor}
\begin{proof}
    Note that realizability is conjugacy-invariant.  Conjugacy classes in $O(3,\Z)$ correspond to signed cycle data: an underlying permutation $\sigma\in S_3$ with $k$ cycles, and signs $\{\epsilon_i\}_{i=1}^k$, where $\epsilon_i\in\F_2$ gives the parity of the number of $(-1)$s on each cycle.  Then by \Cref{thm:large_rank}, \Cref{constr:5}, and the standard-linear actions of \Cref{sec:cyclic}, it suffices to show that the matrix \[\begin{pmatrix}
        0&-1&0\\0&0&-1\\-1&0&0
    \end{pmatrix}\] is realizable. To do this, consider $S^4\subseteq \R^5$ equipped with a standard basis $e_1,\dots,e_5$, and with three marked points $p_1,p_2,p_3\in S^4$ placed in an equilateral triangle in the $e_1,e_2$-plane.  There is then an action on $S^4$ given by rotating in the $e_1,e_2$-plane by $2\pi/3$, composed with reflections in the vectors $e_3$ and $e_4$. Then glue three copies of $\cp^2$ equivariantly at $p_1,p_2,p_3$, such that the reflections in $e_3,e_4$ extend together    to perform complex conjugation on each copy.  (In fact, this construction extends to realize the element $-Id\cdot (12\dots n)\in O(H_{M_n})$ for any $n$.)
\end{proof}

\subsection{Proving \Cref{thm:odd-ab}}\label{sec:odd-ab}

\begin{lem}\label{lem:2-group-gen}
    There is a constant $C>0$ so that if $P\leq O(H_{M_n})$ is a 2-subgroup realized by a group $\tilde P\leq \diff^+(M_n)$, then $d(P)\leq C\log n$, where $d(P)$ is the size of a minimal generating set for $P$.
\end{lem}

\begin{proof}
    Let $\abs{P}=2^m$ be realizable in $\diff^+(M_n)$, and let $\phi\in P$ be a central involution, realized by $\tilde \phi\in\diff^+(M_n)$.  By \Cref{lem:no-free-involution}, $\tilde\phi$ does not act freely on $M_n$.  Since $\phi$ is central, $\tilde P$ acts on $M_n^{\ang{\tilde\phi}}$, and so by \Cref{prop:ed}, \[\beta_0(M_n^{\ang{\tilde\phi}})\leq \beta_0(M_n^{\ang{\tilde\phi}})+\beta_2(M_n^{\ang{\tilde\phi}})=t+2\leq n+2.\]  Then $M_n^{\ang{\tilde\phi}}$ has at most $n+2$ components, and so there is a subgroup $P_0\leq P$ of index at most $n+2$ such that $\tilde{P_0}$ stabilizes a connected component $F\leq M_n^{\ang{\tilde\phi}}$.  If $F$ is a point then $P_0$ embeds into $SO(4)$ and hence is 4-generated because it is a 2-group.
    
    If $F$ is a surface, then let $\tilde K\leq\tilde{P_0}$ be the kernel of the action of $\tilde {P_0}$ on $F$.  Then $K$ acts faithfully on the normal bundle of $F$, and so $K$ embeds into $SO(2)$ and is cyclic.  Since $\tilde P_0/\tilde K$ acts faithfully on $F$, then \[\abs{P_0/K}\leq C'\beta_1(S)\leq C'n\] for some constant $C'> 0$, when $F$ is hyperbolic.  The first inequality follows for high genus by \Cref{thm:84(g-1)} (passing to the orientation double cover if $S$ is non-orientable), and the second inequality follows by \Cref{prop:ed} since $n=t+c+2r$.  Then since $K$ is cyclic, \[d(P_0)\leq d(K)+d(P_0/K)\leq 1+\log_2\abs{P_0/K}=O(\log n).\]  When $F$ is $S^2, T^2, \rp^2$, or $N_2$, then $d(P_0/K)$ is uniformly bounded by uniformizing to a round or flat metric, and so $d(P_0)=O(\log n)$ trivially in this case.  Since $[P:P_0]\leq n+2$ and $P$ is a 2-group, $P$ can be generated by $P_0$ and at most $\lceil\log_2(n+2)\rceil$ additional generators.  So $d(P)=O(\log n)$.
\end{proof}

\begin{lem}\label{lem:count-by-cohomology}
    The number of subgroups $H\leq O(H_{M_n})$ realizable in $\diff^+(M_n)$ is at most \[2^{n^2/16+o(n^2)}.\]
\end{lem}

\begin{proof}
    Let $H\leq O(H_{M_n})$ be realizable in $\diff^+(M_n)$.  Let $\rho:O(H_{M_n})\to S_n$ be the sign-forgetting map.  Let $K=H\cap G_n=\ker(\rho|_H)$, and let $Q=\rho(H)\leq S_n$.  By \Cref{lem:2-group-gen}, $\log_2\abs{K}\leq C \log n$ for a constant $C>0$. Further, $K$ is an elementary abelian 2-group.  Then the total number of possibilities for $K$ is at most \[\sum_{j\leq C\log n}\binom{n}{j}_2\leq 2^{O(n\log n)}.\] By \cite{sn, Theorem 1}, the number of possibilities for $Q$ is at most the number of subgroups of $S_n$, which is \[\abs{\text{Sub}(S_n)}=2^{n^2/16+o(n^2)}.\]  Now fix $K$ and $Q$; we count the number of possibilities for $H$.  Let $M=G_n/K$; then $H/K\leq M\rtimes Q$, intersecting $M$ trivially and projecting isomorphically onto $Q$.  In other words, $H/K$ is given as the graph of a crossed homomorphism, a 1-cocycle $\alpha\in Z^1(Q,M)$.  Then the number of possible $H$ is bounded by $\abs{Z^1(Q,M)}$.
    
    Let $P_H\leq H$ be a Sylow-2 subgroup.  Then $P_Q=\rho(P_H)\leq Q$ is also a Sylow-2 subgroup, and so by \Cref{lem:2-group-gen}, \[d(P_Q)\leq d(P_H)\leq C\log n.\]  The restriction map $\text{res}:H^1(Q,M)\to H^1(P_Q,M)$ is injective: the composition of restriction and corestriction is multiplication by the index $[Q:P_Q]$.  Since $P_Q$ is a Sylow-2 subgroup, this index is odd, and hence the composition acts as the identity on the $\F_2$-vector space $H^1(Q,M)$.  Then \[\dim_{\F_2}H^1(Q,M)\leq \dim_{\F_2}H^1(P_Q,M).\]  Since a 1-cocycle is determined by its values on generators, \[\dim_{\F_2}Z^1(P_Q,M)\leq d(P_Q)\dim_{\F_2}(M)\leq Cn\log n.\]  Then \[\dim H^1(Q,M)\leq Cn\log n.\]  Note also that $\dim B^1(Q,M)\leq\dim M\leq n$.  As $H^1(Q,M)=Z^1(Q,M)/B^1(Q,M)$, then $\dim Z^1(Q,M)\leq O(n\log n)$, and so \[\abs{Z^1(Q,M)}\leq 2^{O(n\log n)}.\]  Multiplying the counts for $K,Q$ and the cocycle gives that the number of realizable groups in $O(H_{M_n})$ is at most \[2^{O(n\log n)}2^{n^2/16 + o(n^2)}2^{O(n\log n)}=2^{n^2/16 + o(n^2)}.\] 
\end{proof}

\begin{proof}[Proof of part 1 of \Cref{thm:odd-ab}]
    \Cref{lem:count-by-cohomology} says that at most $2^{n^2/16 + o(n^2)}$ subgroups of $O(H_{M_n})$ are realizable in $\diff^+(M_n)$.  But even just by considering subgroups of $G_n$ with rank $\lfloor n/2\rfloor$, we find that $O(H_{M_n})$ has at least $\binom{n}{\lfloor n/2\rfloor}_2=2^{n^2/4+O(n)}$ subgroups.  Finally, calculate \[\frac{2^{n^2/16 + o(n^2)}}{2^{n^2/4+O(n)}}=2^{-3n^2/16+o(n^2)}\to 0.\]
    To consider conjugacy classes, note that the number of conjugacy classes of diagonal subgroups of $B_n$ is at least $1/n!$ times the number of diagonal subgroups, giving at least \[2^{n^2/4 - O(n
    \log n)}\] conjugacy classes.  On the other hand, the number of realizable conjugacy classes of subgroups is at most the number of realizable subgroups.  Then by the same calculation, \[\frac{2^{n^2/16 + o(n^2)}}{2^{n^2/4-O(n\log n)}}=2^{-3n^2/16+o(n^2)}\to 0.\]
\end{proof}

\begin{proof}[Proof of part 2 of \Cref{thm:odd-ab}]

    Let $\rho:O(H_{M_n})\to S_n$ be the sign-forgetting map.  Let $\phi\in O(H_{M_n})$ have odd order, and set $\sigma=\rho(\phi)$.  Then $\sigma$ has odd order.  On a cycle $C=(i_1\cdots i_d)$ of $\sigma$, write \[\phi(e_{i_j})=\epsilon_j e_{i_{j+1}},\qquad \epsilon_j\in\{\pm 1\}.\]Then $\phi^d$ acts on the span of this cycle by multiplication by $\prod_j\epsilon_j$.  Hence $\phi$ has odd order precisely when every cycle of $\sigma$ has sign product $+1$, equivalently an even number of minus signs.

    For fixed $\sigma$, the number of odd order signed lifts is therefore \[\prod_{C}2^{|C|-1}=2^{n-c(\sigma)},\]where $c(\sigma)$ is the number of cycles of $\sigma$.  Thus the pushforward of the uniform measure on odd order elements of $O(H_{M_n})$ to $S_n$ is proportional to $2^{n-c(\sigma)}$, i.e. it is exactly $P_{n,1/2}$.

    Moreover, every such lift $\phi$ of $\sigma$ is conjugate in $O(H_{M_n})$ to the ordinary permutation matrix $\sigma$, by changing signs along each cycle.  Since realizability is invariant under conjugacy in $O(H_{M_n})$, a realizable $\phi$ projects to a realizable odd order permutation $\sigma$.  The result now follows from \Cref{cor:prime-cycle-obstruction} and \Cref{prop:theta-prime-cycles} with $\theta=1/2$.
\end{proof}

\begin{rmk}
    The bounds and proof of part 2 of \Cref{thm:asym} carry over with no change to odd order elements of $O(H_{M_n})$, since conjugacy classes of odd order elements in $S_n$ are in bijection with conjugacy classes of odd order elements in $O(H_{M_n})$.  (In fact, this bijection is induced by the projection map $\rho:O(H_{M_n})\to S_n$.  In particular, the asymptotic probability that a random odd order element of $O(H_{M_n})$ is realizable in $\diff^+(M_n)$, up to conjugacy, is at most $e^{-1/2}$. \Cref{conj} can then alternatively be phrased in terms of odd order elements of $O(H_{M_n})$.
\end{rmk}

\noindent In the rest of this subsection, always assume $n$ is odd.  The proof we now present of part 3 of \Cref{thm:odd-ab} is similar in spirit to that of \Cref{thm:large_rank}.

\begin{rmk}
    One simple source of examples for abelian subgroups not contained in either $G_n\leq O(H_{M_n})$ or $S_n\leq O(H_{M_n})$ is the natural inclusion \[\iota_{n,k}: O(H_{M_n})\times O(H_{M_k})\into O(H_{M_{n+k}}),\] by considering $H_1\times H_2$ for abelian subgroups $H_1\leq O(H_{M_n})$ and $H_2\leq O(H_{M_k})$. Even though these maps are group-theoretically basic, the Nielsen realization problem behaves chaotically with respect to the inclusion maps $\iota_{n,k}$.  Sometimes for $\sigma\in O(H_{M_n})$, $\sigma$ is not realizable in $\diff^+(M_{n})$ but $\iota(\sigma,Id)$ is realizable in $\diff^+(M_{n+k})$ for $k$ sufficiently large \cite{ht}.  On the other hand, \Cref{thm:large_rank} implies that for $H_1\leq G_n$ and $H_2\leq G_k$ both realizable, $\iota_{n,k}(H_1,H_2)$ is often not realizable in $\diff^+(M_{n+k})$.
\end{rmk}

\begin{lem}\label{lem:x-blocks}
    Let $n\geq 1$ be odd, and let $H\leq O(H_{M_n})$ be abelian.  Suppose $1\neq \phi,\psi\in H$ are distinct with $\phi^2=\psi^2=1$.  Then at least one of $\phi, \psi,$ or $\phi\psi$ has an odd number of $+1$s along its diagonal as a signed permutation matrix in $O(H_{M_n})$.
\end{lem}

\begin{proof}
    Let $K=\ang{\phi,\psi}\leq (\Z/2\Z)^2$, and let $\bar K$ be its image under the projection $O(H_{M_n})\to S_n$.  For each $g\in K\setminus\{1\}$, let $s(g)\in\F_2$ be the parity of the number of $+1$s along the diagonal of $g$ as a matrix.  Then the set $\{1,2,\dots,n\}$ decomposes into $\bar K$-orbits, which all have size either 1, 2, or 4.  Consider the quantity $s(K)\coloneqq s(\phi)+s(\psi)+s(\phi\psi)\in \F_2$.  An orbit of size 4 contributes 0 to $s(K)$, since no nontrivial element of $K$ fixes any element of that orbit.  An orbit of size 2 also contributes 0 to $s(K)$, since such an orbit is pointwise fixed by exactly one nontrivial element $h$. Since $K$ is abelian and conjugation by any $g\in K$ with $g\neq h$ swaps the two orbit points, then $s(h)=0$.  A singleton orbit is fixed pointwise by all of $K$.  Either all of $\phi,\psi,\phi\psi$ have a corresponding $-1$ on their diagonals, or only one of these elements does.  In either case, such an orbit contributes 1 to $s(K)$.

    Since $n$ is odd, then the number of singleton orbits is odd.  Thus $s(K)=1$.  Then at least one of $s(\phi),s(\psi),s(\phi\psi)$ is 1, as desired.
\end{proof}

\noindent Say an abelian subgroup $H\leq O(H_{M_n})$ \textit{has an odd involution} if there is some element $\phi\in H$ with $\phi^2=1$ and an odd number of $+1$s on its diagonal.  Then \Cref{lem:x-blocks} implies that subgroups without an odd involution have at most one element of order 2.  Equivalently, they have cyclic 2-primary parts.

\begin{lem}\label{lem:cyclic-sylow}
    Let $n\geq 1$, and let $C_n$ be the number of abelian subgroups of $S_n$ with 3-generated Sylow-2 subgroup.  Then \[C_n\leq 2^{(\log_2 3)n^2/9+O(n\log n)}.\]
\end{lem}

\begin{proof}
    Fix some such $B\leq S_n$.  Let $m_1,\dots, m_r$ be the sizes of its orbits. Since $B$ is abelian, there is an inclusion \[B\into H_1\times H_2\times\dots\times H_r\eqqcolon M,\] a product of regular abelian groups of order $m_i$ which record the action on each orbit.

    We now bound the number of subgroups of $S_n$ of the form $M$.  There are at most $n^n$ partitions of $n$, and at most $m_i$ abstract abelian groups of order $m_i$.  Once such an abstract group $H_i$ has been chosen, there are at most $m_i!$ particular identifications.  Then the number of $M$ is bounded by $n^n\prod_{i=1}^r(m_i!\cdot m_i)\leq n^nn!3^{n/3}=2^{O(n\log n)}$.

        Next, we bound the number of subgroups of some fixed $M$.  Decompose $M=M_2\times M_{odd}$ for $M_2$ the 2-primary part.  Any subgroup $B\leq M$ satisfying the conditions of the lemma similarly decomposes as \[B=B_2\times B_{odd}\] with $B_2\leq M_2$ and $B_{odd}\leq M_{odd}$.  Since $B_2$ is at most 3-generated by hypothesis, there are at most $\abs{M}^3=(\prod_i m_i)^3 < (3^{n/3})^3=3^n$ possibilities for $B_2$.  In $M_{odd}$, any subgroup is generated by at most $\log_3(\abs{M_{odd}})\leq 3^{n/3}\leq n/3$ elements.  Thus the number of subgroups of $M_{odd}$ is at most \[\sum_{j=0}^{\lfloor n/3\rfloor}\abs{M_{odd}}^j\leq (n+1)\abs{M_{odd}}^{n/3}\leq (n+1)(3^{n/3})^{n/3}=3^{n^2/9+O(\log n)}.\]  So for each fixed $M$, the number of subgroups $B$ is bounded by $2^{O(n)}3^{n^2/9+O(\log n)}=2^{\log_2(3)n^2/9+O(n)}$.  Multiplying by the number of $M$ gives the result.
\end{proof}

In the following proofs, for $H$ a finite abelian group, let $H_2$ denote its 2-primary part, and let $H_{odd}$ denote the product of its $p$-primary parts for odd primes $p$.

\begin{lem}\label{lem:bounded-lifts}
    Let $n\geq 1$, and let $B\leq S_n<O(H_{M_n})$ be abelian with 3-generated 2-primary part.  Then the number of $H\leq O(H_{M_n})$ with 3-generated 2-primary part and with $\bar H=B\leq S_n<O(H_{M_n})$ is at most $2^{7n}$.
\end{lem}\
\begin{proof}
    Count lifts $B_2$ and $B_{odd}$ separately, noting that $\overline{H_2}=B_2$ and $\overline{H_{odd}}=B_{odd}$.  Since $B_2$ is 3-generated, and the projection $\rho:O(H_{M_n})\to S_n$ has kernel of size $2^n$, then $\abs{\rho\inv(B_2)}\leq2^n\abs{B_2}\leq 2^n\cdot 2^n=2^{2n}$.  Then the number of lifts of $B_2$ is bounded by $(2^{2n})^3=2^{6n}$.  Applying Schur-Zassenhaus to \[1\to\ker\rho\to\rho\inv(B_{odd})\to B_{odd}\to 1,\] all complements lifting $B_{odd}$ are conjugate by $\ker\rho$.  There are then at most $\abs{\ker \rho}=2^n$ lifts of $B_{odd}$.  Thus, the total number of lifts of $B$ is at most $2^{7n}$.
\end{proof}

\noindent Finally, we assemble these three lemmas together to prove that abelian subgroups of $O(n,\Z)$ are asymptotically rarely realizable for odd $n$.
\begin{proof}[Proof of part 3 of \Cref{thm:odd-ab}]
    By considering only the diagonal subgroups, notice that $O(H_{M_n})$ has at least $\binom{n}{\lfloor n/2\rfloor}_2=2^{n^2/4+O(n)}$ abelian subgroups.  Then \Cref{lem:x-blocks}, \Cref{lem:cyclic-sylow}, and \Cref{lem:bounded-lifts} together imply that the number of abelian subgroups which either lack an odd involution or have 3-generated 2-primary part is bounded above by $2^{(\log_2 3)n^2/9+O(n\log n)}$.  Then as $n$ grows, these constitute a vanishing proportion of all abelian subgroups.

    Now recall that any abelian group $H\leq SO(4)$ is 3-generated.  Thus as $n$ grows, almost all abelian subgroups $H\leq O(H_{M_n})$ satisfy both of \begin{enumerate}
        \item $H$ has an odd involution, and
        \item $H_2$ does not embed into $SO(4)$.
    \end{enumerate}
    For contradiction, suppose such an $H$ satisfying both criteria is realizable.  Then by \Cref{prop:ed}, the order 2 element of $H$ with an odd number of diagonal $+1$s fixes an odd number of isolated fixed points.  Then $H_2$ fixes at least one of these points $p$ since it is a 2-group.  But $H_2$ does not embed into $SO(4)$, so it cannot act faithfully on $T_pM_n$, which is a contradiction.
\end{proof}

\printbibliography

\end{document}

%% file: header.tex
\usepackage{fullpage,  amssymb, amsthm, amsmath, amsfonts, times}
\usepackage{thmtools}

\usepackage{graphicx}
\usepackage{epstopdf}
\usepackage[style=alphabetic]{biblatex}
\usepackage{enumitem}
\usepackage{mathrsfs}
\usepackage{mathtools}
\usepackage{centernot}
\usepackage{xfrac}
\usepackage{soul}
\usepackage{fullpage}
\usepackage[normalem]{ulem}
\usepackage{eufrak}
\usepackage[symbol]{footmisc}
\usepackage{stmaryrd}
\usepackage{makecell}
\usepackage{bm}
\usepackage[all,cmtip]{xy}
\usepackage[usenames,dvipsnames]{xcolor}
\usepackage[margin=1.0 in]{geometry}
\usepackage{float}
\usepackage{hyperref}
\usepackage{cleveref}
\hypersetup{colorlinks=true, citecolor=blue, linkcolor=blue, urlcolor=blue}
\usepackage{array}
\usepackage[ruled,vlined]{algorithm2e}
\usepackage{changepage}

\usepackage{tikz}
\usepackage{comment}
\usetikzlibrary{positioning,arrows}
\usetikzlibrary{calc}
\usetikzlibrary{matrix}

\usetikzlibrary{decorations.pathreplacing}

\tikzset{
  mybrace/.style={
    blue,
    line width=1.25pt,
    decorate,
    decoration={brace,amplitude=7pt}
  }
}

\theoremstyle{plain} 
\newtheorem{theorem}{Theorem}[section]
\newtheorem*{theorem*}{Theorem}
\newtheorem*{obs*}{Observation}
\newtheorem{prop}[theorem]{Proposition}

\newtheorem{lem}[theorem]{Lemma}

\newtheorem*{convention*}{Convention on Randomness}
\newtheorem{cor}[theorem]{Corollary}
\newtheorem{conj}[theorem]{Conjecture}

\theoremstyle{remark}

\newtheorem{rmk}{Remark}

\theoremstyle{definition}

\crefalias{def}{definition}
\newtheorem{eg}{Example}
\newtheorem{constr}{Construction}

\newtheoremstyle{exstyle}
  {6pt}   
  {6pt}   
  {\normalfont} 
  {}      
  {\bfseries} 
  {.}     
  {\newline} 
  {}      

\renewcommand{\mod}[1]{{\ifmmode\text{\rm\ (mod~$#1$)}\else\discretionary{}{}{\hbox{ }}\rm(mod~$#1$)\fi}}

\newcommand{\floor}[1]{\left\lfloor#1\right\rfloor}

\newcommand{\abs}[1]{\left|#1\right|}

\newcommand{\aut}{\textup{Aut}}

\renewcommand{\ker}{\textup{ker }}

\newcommand{\im}{\textup{Im}}
\newcommand{\re}{\textup{Re}}

\renewcommand{\bar}{\overline}

\newcommand{\mpg}{\textup{Mod}}

\newcommand{\into}{\hookrightarrow}

\newcommand{\rp}{\mathbb{RP}}
\newcommand{\cp}{\mathbb{CP}}
\newcommand{\diff}{\textup{Diff}}
\newcommand{\de}{\textup{def}}

\newcommand{\homeo}{\text{Homeo}}

\newcommand{\isom}{\text{Isom}}

\newcommand{\ang}[1]{\langle#1\rangle}

\newsymbol\dnd 232D

\renewcommand{\t}{\theta}

\renewcommand{\emptyset}{\varnothing}

\newcommand{\C}{\mathbb C}

\newcommand{\F}{\mathbb F}

\newcommand{\Q}{\mathbb Q}
\newcommand{\R}{\mathbb R}
\newcommand{\Z}{\mathbb Z}

\newcommand{\inv}{^{-1}}

\newcommand{\rk}{\text{Rk}}

\newcommand{\0}{{\vec 0}}

\newcommand{\ignore}[1]{}
